\documentclass[tikz, border=5pt,11pt,leqno]{article}

\usepackage{amsmath, amsfonts,color,xcolor,
}

\usepackage{bbm}
\usepackage[latin1]{inputenc}
\usepackage[T1]{fontenc}
\usepackage{xcolor}

\usepackage{graphicx} 

\usepackage[english]{babel}

\usepackage{lmodern}

\usepackage{amsmath, amssymb, amsthm}
\usepackage{appendix}

\usepackage{mathrsfs}
\usepackage{latexsym}
\usepackage{graphicx}
\usepackage{cancel}
\usepackage[normalem]{ulem}

\usepackage{tikz}
\usetikzlibrary{patterns, arrows.meta, calc}

\let\d=\delta

  \let\L=\Lambda

\newcommand{\re}{{\mathbb R}}

\def\R{{\mathbb R}}
\def\N{{\mathbb  N}}
\def\N{{\mathbb N}}

\newcommand{\n}{{{I\!\!Z}}}

\let\d=\delta

\newcommand{\cP}{{\mathcal P}}
\newcommand{\GammamuoptC}  {\Gamma^{\mu,{\rm opt}}_C}

\newtheorem{theorem}{Theorem}[section]
\newtheorem{lemma}{Lemma}[section]
\newtheorem{proposition}{Proposition}[section]
\newtheorem{definition}{Definition}[section]
\newtheorem{corollary}{Corollary}[section]
\newtheorem{remark}{Remark}[section]
\newtheorem{example}{Example}[section]

  {\hfill$\Box$\bigskip\par}

\DeclareMathOperator{\sign}{sign}

\tikzset{
  >=Latex,
  axis/.style={thick},
  mainline/.style={thick},
  hatch/.style={pattern=north east lines,pattern color=cyan!60},
  dot/.style={circle,inner sep=1.6pt,fill}
}

\begin{document}

\title{Constrained Mean Field Games with Grushin type dynamics}
\author{{\large \sc Alessandra Cutr\`i\thanks{Dipartimento di Matematica, Universit\`a di Roma Tor Vergata, cutri@mat.uniroma2.it}, Paola Mannucci\thanks{Dipartimento di Matematica ``Tullio Levi-Civita'', Universit\`a di Padova, mannucci@math.unipd.it}, Claudio Marchi \thanks{Dipartimento di Matematica ``Tullio Levi-Civita'', Universit\`a di Padova, claudio.marchi@unipd.it}, Nicoletta Tchou \thanks{Univ Rennes, CNRS, IRMAR - UMR 6625, F-35000 Rennes, France, nicoletta.tchou@univ-rennes1.fr}}} 
\date{\today}

\maketitle

\begin{abstract}
This paper is devoted to a class of finite-horizon deterministic mean field games with Grushin-type dynamics, state constraints, and nonlocal coupling.  First, we consider the optimal control problem that each agent aims to solve when the evolution of the population is given, and we establish some properties: the existence of an optimal trajectory for any starting point~$(x,t)$, the closed graph property for the multivalued map which associates to each point $(x,t)$ the set of optimal trajectories starting from that point, endowed with a suitable notion of convergence, and the continuity of the value function. The main issue to overcome is due to the local interplay at boundary points between the set of state constraints and the degenerate dynamics. To this end, we shall point out two different sets of assumptions which are both sufficient for these properties.

Afterwards, we tackle the mean field games; taking advantage of the aforementioned properties, we prove the existence of a relaxed equilibrium (which describes the evolution of the game in terms of a probability on the set of admissible trajectories) and derive the existence of a mild solution (which is a pair formed by the value function for the generic player and a family of time-dependent measures on the state).
\end{abstract}

\section{Introduction}

In this paper we shall deal with a class of finite-horizon deterministic mean field games with Grushin-type dynamics, state constraints, and nonlocal coupling.  First, we consider the optimal control problem that each agent aims to solve when the evolution of the population is given. In this finite-horizon optimal control problem with state constraints, the player aims to minimize the cost in~\eqref{eq:46_intro} below by choosing a trajectory which must fulfill the state constraints and obey the Grushin-type dynamics~\eqref{grugen_intro} with $\nu\in(0,\infty)$.
We focus on Grushin-type dynamics because they are a prototype of dynamics where, at some points of the state space, the agent cannot move in some directions  (here, it cannot move in the $x_2$-direction when it is on the $x_2$-axis).
More precisely, we fix a time horizon $T>0$ and a compact subset $\Sigma\subset \mathbb R^2$ and we consider the control problem where the admissible curves $y(s)=(y_1(s),y_2(s))$, starting from point~$x\in \Sigma$ at time~$t\in[0,T]$, must fulfill $y(t)=x$, $y(\cdot)\in \Sigma$ in~$[t,T]$, $y\in W^{1,1}(t,T)$, and
\begin{equation}\label{grugen_intro}
\left\{\begin{array}{lll}
\dot y_1(s)=\alpha_1(s)\\[0.5mm]
\dot y_2(s)=|y_1(s)|^\nu\alpha_2(s)
\end{array}\right.
\end{equation}
for a.e.\, $s\in[t,T]$, for some measurable function $\alpha(s)=(\alpha_1(s),\alpha_2(s))$.

Clearly, these dynamics are related to the set of vector fields $X=\{X_1,X_2\}$ with $X_1=(1,0)$ and $X_2=(0,|x_1|^\nu)$. The study of such a system of vector fields goes back to the pioneering papers by Baouendi~\cite{Ba67} and Grushin~\cite{Gru70}. It is worth noting that the system~$X$ is not of H\"ormander type if $\nu\ne 2n$ for some $n\in\mathbb N\setminus\{0\}$.
Nevertheless, as pointed out in~\cite{FL1}, the Carnot-Carath\'eodory space associated with~$X$ still enjoys several nice properties: the balls in the Carnot-Carath\'eodory distance~$d_{CC}$ are open in the Euclidean topology (more precisely, for some positive constants~$C$ and~$a$, there holds $d_{CC}(x,y)\leq C |x-y|^a$), and the doubling property holds true (namely, for some $C_1>0$, there holds $|B_{CC}(x,2r)|\leq C_1|B_{CC}(x,r)|$, where $B_{CC}(x,r)$ is the $d_{CC}$-ball centered at~$x$ with radius~$r$). It is also worth noting that the second-order operator of this system,  $X_1^2+X_2^2$, has been intensively studied because it still enjoys many properties of more regular operators (see \cite{FL1,FGW} and references therein).

The cases of standard Grushin dynamics (which fulfill the H\"ormander condition) and  Euclidean dynamics, namely, when $X_2=(0,x_1^\nu)$ for $\nu\in \N\setminus\{0\}$ and  when $X_2=(0,1)$, respectively,  can be dealt with using the same arguments as in this paper with easy adaptations.

The agent at point $x\in\Sigma$ at time~$t\in[0,T]$ will choose its trajectory~$(y,\alpha)$ among the admissible ones so as to minimize the cost
\begin{equation}\label{eq:46_intro}
J_t(x;(y,\alpha))=\int_t^T\left[\frac{|\alpha(\tau)|^2}{2}+\ell(y(\tau),\tau)\right] d\tau+g(y(T)).
\end{equation}
Clearly, an optimal trajectory $(y,\alpha)$ must satisfy the properties: $(a)$ $y\in \Sigma$, $(b)$  $(y,\alpha)$ fulfills the Grushin-type dynamics~\eqref{grugen_intro}, and $(c)$  $(y,\alpha)$ minimizes the cost~\eqref{eq:46_intro} (in particular, $\alpha\in L^2(t,T)$).

We shall investigate some properties of this optimal control problem: the existence of an optimal trajectory for any starting point~$(x,t)\in\Sigma\times [0,T]$, the closed graph property for the multivalued map which associates to each point $(x,t)$ the set of optimal trajectories starting from that point, endowed with a suitable notion of convergence, and the continuity of the value function.
The existence of an optimal trajectory for each starting point is a consequence of the structure of  the cost. In particular, in our framework the cost functional $J_t$  will be lower semicontinuous with respect to $\alpha$  in the $L^2[t,T]$-weak topology because the running cost is separated and strictly convex in the control variable and the running and  final costs are continuous.

On the other hand, the closed graph property and  continuity of the value function depend on the interplay between the geometry of~$\Sigma$ and the structure of the dynamics. In particular, since our dynamics do not satisfy a strong controllability assumption, the continuity of the value function is not a priori guaranteed by the now  classical {\it inward pointing condition}; see~\cite{Soner}.
Our main result is to point out two different sets of assumptions on the local interplay between $\Sigma$ and the dynamics which are both sufficient for the closed graph property and   continuity of the value function. One assumption requires  a suitable approximation property of any admissible trajectory starting from the point~$x$ and, in turn, it is ensured by the property that the point~$x$ can be reached by an admissible trajectory starting from any point $z\in\Sigma$ nearby with a cost that vanishes as $z$ converges to~$x$ (for similar results, under stronger assumptions, see~\cite{BFZ}). The other assumption requires that the point~$x$ cannot be reached by any admissible trajectory with finite cost. It is worth pointing out that neither of these two assumptions requires that the set~$\Sigma$ is the closure of a bounded open domain with $C^2$ boundary as in~\cite{CC}; namely, in our results, $\Sigma$ may have empty interior and a boundary not even~$C^1$.

Without state constraints, we refer the reader to \cite[Chapter IV]{BCD} and to \cite[Part~$2$]{Cor} for classical results for control problems with degenerate dynamics satisfying  a Chow/H\"ormander condition.

As an application of these properties, we shall tackle a mean field game problem (MFG for short), with a finite time horizon, where the agents must remain in~$\Sigma$, follow trajectories obeying   the dynamics~\eqref{grugen_intro}, and aim to minimize the cost~\eqref{eq:46_intro} where the costs~$\ell$ and~$g$ depend in a nonlocal, regularizing  manner on the distribution of agents.

Let us recall that the theory of MFGs was introduced by Lasry and Lions \cite{LL1,LL2,LL3} and it has developed widely in the last two decades.
It studies the  asymptotic behavior of  deterministic, or stochastic, differential games with many, indistinguishable, players as the number  of players tends to infinity. The description of the literature on this topic  goes far beyond the aim of the present paper; for an overview, we refer the reader to the notes by Cardaliaguet~\cite{C} on the lectures by Lions and by Cardaliaguet and Porretta~\cite{CP}. Most of the literature on deterministic MFGs deals with  strongly controllable dynamics without state constraints.

With regard to the state-constrained deterministic MFGs, an important difficulty occurs: The density of the population may concentrate on the boundary of the state constraint even if the initial distribution is absolutely continuous. This phenomenon was first observed in some applications of MFGs to macroeconomics (see \cite{AHLLM}) and makes it difficult to  characterize the state distribution by means of a differential equation, as in the unconstrained case.
Such an issue was tackled in~\cite{CC} where, using some ideas of~\cite{BB, BC15, CM16},  the authors introduced a relaxed notion of equilibrium, which is defined in a  Lagrangian setting rather than with PDEs. Since the set of optimal trajectories starting from a given point may not be a singleton, the relaxed solutions of the MFG problem are probability measures on a suitable set of  admissible trajectories. Once the existence of a relaxed equilibrium was ensured, in~\cite{CC}, the authors associated with each relaxed equilibrium a {\it mild} solution which is a pair $(u,m)$ where $m=(m(t))_t$ is a suitable family of time-dependent probability measures on the state (obtained by a natural push-forward of the equilibrium) while the function~$u$ is the value function of the optimal control problem associated with~$m$.

With regard to deterministic MFGs with dynamics that are not strongly controllable without state constraints, we quote here the papers \cite{MMMT,MMT} where we study some cases with degenerate dynamics including the Grushin dynamics. In these cases we have more regularity of the solution and we are able to  follow a PDE approach.
In the present paper, following this Lagrangian approach, our aim is to prove the existence of  relaxed MFG equilibria
which are described by  probability measures defined on a  set  of admissible trajectories.
The proof  involves Kakutani's fixed point theorem,  applied to a multivalued map defined on a suitable,  convex, and compact  set of  probability measures on the set of admissible trajectories. In order to apply   Kakutani's theorem, a key role is played by the closed graph property established for the control problem.
Finally, to any relaxed MFG equilibrium, we can associate a {mild solution} as in~\cite{CC}.

This paper is organized as follows: In the rest of this section, we fix some notation  and introduce the control problem and its assumptions. Section~\ref{sect:cgp} is devoted to establishing the closed graph property for the multivalued map which associates to each point the set of optimal trajectories starting from that point. In particular, we shall point out two different sets of sufficient conditions for this property. One of these sets relies on the {\it approximation property} which, in turn, is ensured by the {\it reachability property}. The other set of conditions relies on the {\it unreachability property}. Section~\ref{sectionsufficient} contains quite general assumptions that guarantee the reachability property. Section~\ref{sect:vf} is devoted to introducing the value function associated with our control problem and studying its regularity.
In particular, we point out some sufficient conditions for its continuity.

In Section~\ref{sect:ex}, we provide several examples of control problems where the closed graph property and/or the reachability property hold or do not hold. Finally, in Section~\ref{sect:MFGequil}, taking advantage of the previous results, we tackle a class of MFG problems with Grushin-type dynamics, state constraints, and nonlocal couplings.

\subsection{Setting  of the control problem}
\label{sect:settings}

Let $T>0$ be a finite time horizon and let $\Sigma$ be a compact subset of $\mathbb R^2$.
For $\nu\in(0,\infty)$ fixed, consider the control problem where the admissible curves
$y(s)=(y_1(s),y_2(s))$, $s\in[t,T]$, starting from the point $x\in\Sigma$ at time $t\in[0,T]$, i.e., $y(t)=x$, must remain in~$\Sigma$, i.e., $y\in\Sigma$, and obey  the Grushin-type dynamics, i.e., $y\in W^{1,1}(t,T)$ and, for a.e.\, $s\in[t,T]$,
\begin{equation}\label{grugen}
\left\{\begin{array}{ll}
\dot y_1(s)=\alpha_1(s)\\[0.5mm]
\dot y_2(s)=|y_1(s)|^\nu\alpha_2(s)
\end{array}
\right.
\end{equation}
for some measurable function $\alpha(s)=(\alpha_1(s),\alpha_2(s))$.
Note that, given an admissible curve~$y$, the control~$\alpha$ might not be uniquely determined by \eqref{grugen}, even on a set of positive measure.
Indeed, in the set $I_0:=\{s\in[t,T]\:\mid\: y_1(s)=0\}$ any measurable function~$\tilde{\alpha}= (0,\tilde{\alpha}_2(s))$  is an admissible control that verifies \eqref{grugen} in $I_0$.
We define a {\it trajectory} as a pair $(y,\alpha)$, formed by an admissible curve~$y$ and a control~$\alpha$ fulfilling~\eqref{grugen}.

 To any admissible trajectory $(y,\alpha)$ we associate the cost
\begin{equation}\label{eq:46}
J_t(x;(y,\alpha))=\int_t^T\left[\frac{|\alpha(\tau)|^2}{2}+\ell(y(\tau),\tau)\right] d\tau+g(y(T))
\end{equation}
where the functions~$\ell$ and~$g$ belong respectively to the spaces~$C(\Sigma\times [0,T])$ and $C(\Sigma)$ which are endowed with their natural topology. For simplicity, when $t=0$, we write $J(x;(y,\alpha))$ instead of $J_0(x;(y,\alpha))$.

We denote by $\Gamma_t[x]$ the set of admissible trajectories with $y(t)=x$ and finite cost, namely,
\begin{equation*}
\Gamma_t[x]=\{\textrm{admissible trajectories $(y, \alpha)$ with $y(t)=x$ and $J_t(x;(y,\alpha))<\infty$}\}.
\end{equation*}
Note that $\Gamma_t[x]$ consists of all admissible trajectories $(y,\alpha)$ on~$[t,T]$ with~$\alpha \in L^2(t,T)$; even with this restriction,~$\alpha$ is still not uniquely determined by~$y$.

We denote  the set of optimal trajectories for~$(x,t)$ by
\begin{equation}\label{eq:gamma_opt}
\Gamma^{\rm{opt}}_t[x]=\{(y, \alpha)\in\Gamma_t[x]\;:\; J_t(x;(y,\alpha))=\min_{(\hat y,\hat \alpha)\in \Gamma_t[x]} J_t(x;(\hat y,\hat \alpha))\}.
\end{equation}
For simplicity, when $t=0$ we drop the subscript: $\Gamma[x]=\Gamma_0[x]$ and $\Gamma^{\rm{opt}}[x]=\Gamma_0^{\rm{opt}}[x]$. We also set $\Gamma=\cup_{x\in \Sigma} \Gamma[x]$.
For $[t_1,t_2]\subseteq[0,T]$, we shall denote by $\Gamma_{t_1,t_2}[x]$ the set of admissible trajectories $(y,\alpha)$ restricted to the time interval $[t_1,t_2]$ with $y\in W^{1,2}(t_1,t_2)$ and $\alpha\in L^2(t_1,t_2)$.

Clearly, for every $(y, \alpha)\in \Gamma_t[x]$ (and, in particular, for every $(y, \alpha)\in \Gamma^{\rm{opt}}_t[x]$), the control~$\alpha$ belongs to~$L^2(t,T)$ and the curve $y$ belongs to $W^{1,2}(t,T)$.\begin{remark}\label{prp:ex_OT}
An easy adaptation of the arguments in~\cite[Proposition 2.5]{24SIAM} gives that, for each $(x,t)\in\Sigma\times[0,T]$, the set  $\Gamma^{\rm{opt}}_t[x]$ is not empty. For completeness, we recall that the proof relies on the lower semicontinuity of the cost with respect to the $L^2$-weak convergence of the controls and to the following property: If $\{(y_n,\alpha_n)\}_n$ is a minimizing sequence, then there exists $(y,\alpha)\in\Gamma_t[x]$ such that, after extracting a subsequence, $y_n$ converges uniformly to~$y$ and $\alpha_n$ converges weakly in $L^2$ to~$\alpha$.
\end{remark}

\section{Closed graph property}\label{sect:cgp}
In this section, we study a closed graph property for $\Gamma^{\rm{opt}}[x]$ which will play a crucial role in the rest of the paper. We provide two different sets of assumptions which are both sufficient for such a property. One of these sets relies on the approximation property (see Definition~\ref{def:appr_prop}) which states that, given any point~$x\in\Sigma$ and any sequence $\{x_n\}_n$, with $x_n \in \Sigma$ and $ x_n\to x$ as $n\to\infty$, any admissible trajectory starting from~$x$ can be suitably approximated by a sequence of admissible trajectories starting from~$x_n$. In turn, the approximation property is ensured by the reachability property (see Definition~\ref{def:reach_prop}): the point~$x$ can be reached by an admissible trajectory starting from any point~$z\in\Sigma$ nearby with a cost vanishing as $z$ converges to~$x$. The other set of assumptions requires that the point~$x$ cannot be reached by any admissible trajectory with finite cost (see Definition~\ref{def:unreach_pt}).

We now recall the closed graph property and we introduce the approximation property and the notion of an unreachable point.
\begin{definition}
For any $x\in \Sigma$, we say that $\Gamma^{\rm{opt}}[x]$ has the closed graph property when\\
for any sequence $\{x_n   \}_{n\in\N}$ with $x_n\in\Sigma$ and $x_n\to x$ as $n\to\infty$, and for any $(y_n,\alpha_n)\in\Gamma^{\rm{opt}}[x_n]$ with $y_n$ uniformly convergent to a curve~$y$, there exists a measurable function~$\alpha$ such that~$(y,\alpha)\in \Gamma^{\rm{opt}}[x]$.
\end{definition}

\begin{definition}\label{def:appr_prop}
For any $x\in \Sigma$,  we say that the control problem~\eqref{grugen}-\eqref{eq:46} has the approximation property at $x$ when \\
for any $(y,\alpha)\in\Gamma[x]$ and any sequence $\{x_n\}_{n\in\N}$, with $x_n\in\Sigma$ and $x_n\to x$ as $n\to\infty$, there exists a sequence $\{(y_n,\alpha_n)\}_{n\in\N}$ such that, for any $n\in\N$, $(y_n,\alpha_n)\in\Gamma[x_n]$ and
\begin{equation}\label{eq:lemma1}
\begin{array}{rl}
(i)&\quad \lim_{n\to\infty}\|y_n(\cdot)-y(\cdot)\|_{L^\infty(0,T)}=0\\[0.5mm]
(ii)&\quad\lim_{n\to\infty}\left[\|\alpha_n\|_2^2- \|\alpha\|_2^2\right]=0\\[0.5mm]
(iii)&\quad \lim_{n\to\infty}J(x_n;(y_n,\alpha_n)) = J(x;(y,\alpha)).
\end{array}
\end{equation}
\end{definition}
\begin{remark}\label{rmk:2.1}
Observe that for the cost defined in \eqref{eq:46}, condition {\it (iii)} of \eqref{eq:lemma1} is a consequence of {\it (i)} and {\it (ii)} and the continuity of $\ell$ and $g$. Nevertheless, we prefer to write all three properties to facilitate a possible extension of the theory to more general costs. Furthermore, costs depending separately on the $L^p$ norm of controls could be considered by slightly modifying {\it (ii)} in \eqref{eq:lemma1}.
\end{remark}
\begin{definition}\label{def:unreach_pt}
The point $x\in\Sigma$ is said to be unreachable for the control problem~\eqref{grugen}-\eqref{eq:46} if, for any $\bar x\in \Sigma\setminus\{x\}$, there is no trajectory $(y,\alpha)\in \Gamma[\bar x]$ with  $y(T)=x$.
\end{definition}
\begin{remark}\label{rmk:unreach_ellT}
The preceding definition implies that for any $\bar x\in \Sigma\setminus\{x\}$ and for any $\theta \in [0, T]$ there is no trajectory $(y,\alpha)\in \Gamma[\bar x]$ with  $y(\theta)=x$; otherwise, choosing $\alpha(t)=0$, for $t\in [\theta,T]$, we would have $y(T)=x$.
\end{remark}

\begin{remark}\label{rmk:unreach_ell}
The definition of an unreachable point is independent of the choice of the running cost~$\ell$. This issue will play a crucial role in the study of the
mean field games studied in Section \ref{sect:MFGequil}.
\end{remark}
\begin{proposition}\label{prp:prop1}
Let $x\in\Sigma$. Assume one of the following properties:
\begin{itemize}
\item[(i)] the control problem~\eqref{grugen}-\eqref{eq:46} has the approximation property at~$x$
\item[(ii)] the point $x$ is unreachable for the control problem~\eqref{grugen}-\eqref{eq:46}.
\end{itemize}
Then $\Gamma^{\rm{opt}}[x]$ has the closed graph property.
\end{proposition}
\begin{proof}
Consider a sequence $\{x_n\}_{n\in\N}$ with $x_n\in\Sigma$ and $x_n\to x$ as $n\to\infty$, and a sequence $(y_n,\alpha_n)\in\Gamma^{\rm{opt}}[x_n]$ with $y_n$ uniformly convergent to a curve~$y$.
Since the trajectories $(y_n,\alpha_n)$ are optimal, the controls $\alpha_n$ are uniformly bounded in $L^2(0,T)$.
Indeed, comparing $J(x_n,( y_n,\alpha_n))$ with the functional $J(x_n,(x_n, 0))$ evaluated on the constant curve $y(t)=x_n$, corresponding to $\alpha=0$, we get
$$J(x_n,( y_n,\alpha_n))\leq K$$
where $K\geq \|\ell\|_\infty T+\|g\|_\infty$. Thus,
$\|\alpha_n\|_{L^2}\leq 2\sqrt{K}$.

Hence, possibly passing to a subsequence that we still denote by $(y_n,\alpha_n)$, we may assume that $\alpha_n$ converges weakly to some $\alpha$ in $L^2(0,T)$.
In particular,  $\dot y_{n1}\rightharpoonup \alpha_1$ weakly in $L^2[0,T]$ and $y_{n1}\rightarrow y_1$ uniformly in $[0,T]$.
This implies that  $\dot y_{n2}\rightharpoonup |y_1|^\nu\alpha_2$ weakly in $L^2[0,T]$.

Clearly, $(y,\alpha)$ is an admissible trajectory; namely, it fulfills~\eqref{grugen} and stays in $\Sigma$.

We now claim that $(y,\alpha)$ is optimal, namely,
\begin{equation}\label{eq:claim1}
J(x,(\hat y, \hat \alpha))\geq J(x,(y,\alpha)) \qquad \forall (\hat y, \hat \alpha)\in \Gamma[x].
\end{equation}
In order to prove this relation, we split our arguments according to the assumptions.

\vspace*{4pt}\noindent\text{Case $1$.   Assumption~$(i)$ holds true.} This part of the proof is an adaptation of the proofs of \cite[Proposition 2.19 and Proposition 3.10]{24SIAM} so we just sketch it and we refer the reader to that paper for the details. The approximation property ensures that there exists a sequence $\{(\hat y_n,\hat \alpha_n)\}_{n\in\N}$ where $(\hat y_n,\hat \alpha_n)\in \Gamma[x_n]$ satisfies the properties in~\eqref{eq:lemma1} with $(y,\alpha)$ replaced by $(\hat y, \hat \alpha)$. Since $(y_n,\alpha_n)\in \Gamma^{\rm{opt}}[x_n]$, there holds
\begin{equation*}
J(x_n,(\hat y_n, \hat \alpha_n))\geq J(x_n,(y_n,\alpha_n)).
\end{equation*}
We now study separately the two sides of the previous relation. By property~\eqref{eq:lemma1}-$(iii)$, we have
\begin{equation*}
\limsup_{n}J(x_n,(\hat y_n, \hat \alpha_n))=  J(x,(\hat y, \hat \alpha));
\end{equation*}
on the other hand, by the convexity of the running cost with respect to $\alpha$ and the continuity of $\ell$ and $g$, the cost $J$ is lower semicontinuous, i.e.,
\begin{equation*}
\liminf_{n}J(x_n,(y_n,\alpha_n))\geq  J(x,(y,\alpha)).
\end{equation*}
Combining the last two relations with the previous one, we complete the proof of our claim~\eqref{eq:claim1}.

\vspace*{4pt}\noindent\text{Case $2$. Assumption~$(ii)$ holds true.}
We claim that the only admissible trajectory starting at $x$ has to stay at $x$ for any $s\in [0,T]$. Indeed, assume by contradiction that there exists $(\bar y,\bar \alpha)\in \Gamma[x]$ with $\bar y(T)=\bar x\in\Sigma\setminus\{x\}$. Then, reversing the time and the sign of~$\bar \alpha$, we obtain an admissible trajectory from $\bar x$ to $x$ with finite cost, in contradiction with our assumptions.

With the same notation  as in Case 1, $(y_n,\alpha_n)$ tends (up to subsequences) to an admissible trajectory $(y,\alpha)$ where $y_n\rightarrow y $ uniformly in $[0,T]$ and  $\alpha_n\rightharpoonup \alpha $ weakly in $L^2[0,T]$. Then $y(0)=x$ and  $\alpha\in L^2[0,T]$. The claim implies that $y(s)\equiv x$ for any time $s\in [0,T]$.

To prove the closed graph property, we have to prove that there exists a control law $\alpha^*$  such that $(y,\alpha^*)\in\Gamma^{\rm{opt}}[x]$, and this is true when choosing $\alpha^*\equiv 0$.
\end{proof}
Now we give the definition of the reachability property and we prove in the next lemma that it implies the approximation property.
\begin{definition}\label{def:reach_prop}
For any $x\in \Sigma$, we say that the dynamics~\eqref{grugen} have the reachability property at~$x$ when  for any sequence $\{x_n\}_{n\in\N}$, with $x_n\in\Sigma$ and $x_n\to x$ as $n\to\infty$, there exist two sequences $\{(y_n,\alpha_n)\}_{n\in\N}$ and $\{\delta_n\}_{n\in\N}$ such that $(y_n,\alpha_n)\in\Gamma_{0,\delta_n}[x_n]$ and $\delta_n\in[0,T]$ with
\begin{equation*}
y_n(\delta_n)=x,\qquad \lim_{n\to\infty}\|\alpha_n\|_{L^2(0,\delta_n)} =0 \quad \textrm{and } \lim_{n\to\infty}\delta_n=0.
\end{equation*}
\end{definition}

\begin{lemma}\label{lema:reach_appr}
If the dynamics~\eqref{grugen} have the reachability property at~$x\in\Sigma$, then the control problem~\eqref{grugen}-\eqref{eq:46} has the approximation property at~$x$.
\end{lemma}
\begin{proof}
Consider a trajectory $(y,\alpha)\in\Gamma[x]$. We aim at building the approximating trajectory $(y_n,\alpha_n)$ starting from $x_n$ in two steps. In the former one, we use the trajectory of the reachability property while, in the latter one, we follow the trajectory $(y,\alpha)$ suitably rescaled in time.

For any $x_n$, denote by  $(\tilde y_n,\tilde\alpha_n)$ the trajectory as in Definition~\ref{def:reach_prop}. We define
$\alpha_n:[0,T]\to \R^2$ as
\[
\alpha_n(s)=\left\{
\begin{array}{ll}
\tilde\alpha_n(s)&\quad\textrm{for }s\in[0,\delta_n],\\
\frac{T}{T-\delta_n}\alpha \left((s-\delta_n)\frac{T}{T-\delta_n}\right) &\quad \textrm{for }s\in[\delta_n,T]
\end{array}
\right.
\]
and $y_n$ as the corresponding curve:
\[
y_n(s)=\left\{
\begin{array}{ll}
\tilde y_n(s)&\quad\textrm{for }s\in[0,\delta_n],\\[0.5mm]
y\left((s-\delta_n)\frac{T}{T-\delta_n}\right)&\quad \textrm{for }s\in[\delta_n,T].
\end{array}
\right.
\]
Then the trajectory $(y_n,\alpha_n)$ is admissible (with $y_n(T)=y(T)$).

Let us now estimate~$|y(s)-y_n(s)|$ for $s\in[\delta_n,T]$,
\begin{equation*}
\begin{array}{ll}
|y(s)-y_n(s)|\!\!\!\!&= |y(s)- y\left((s-\delta_n)\frac{T}{T-\delta_n}\right)|\\[1.5mm]
&\leq (1+ \textrm{diam}(\Sigma)^\nu)\int_{(s-\delta_n)\frac{T}{T-\delta_n}}^s|\alpha(\tau)|\, d\tau\\[2mm]
&\leq (1+\textrm{diam}(\Sigma)^\nu)\|\alpha\|_2\sqrt{\delta_n}\sqrt{\frac{T-s}{T-\delta_n}}\\[1.5mm]&\leq (1+\textrm{diam}(\Sigma)^\nu)\|\alpha\|_2 \sqrt{\delta_n}.
\end{array}
\end{equation*}
Hence, point $(i)$ in Definition~\ref{def:appr_prop} is completely proved.

In order to prove point~$(ii)$ in Definition~\ref{def:appr_prop}, we write
\begin{eqnarray*}
&&\|\alpha_n\|^2_{L^2(0,T)}- \|\alpha\|^2_{L^2(0,T)}\\&&= \|\alpha_n\|^2_{L^2(0,\delta_n)}+ \frac{T^2}{(T-\delta_n)^2}\int_{\delta_n}^T\left|\alpha((s-\delta_n)\frac{T}{T-\delta_n})\right|^2ds- \|\alpha\|^2_{L^2(0,T)}\\
&&=  \|\alpha_n\|^2_{L^2(0,\delta_n)}+ \left(\frac{T}{T-\delta_n}-1\right) \|\alpha\|^2_{L^2(0,T)} \\&&=o(1).
\end{eqnarray*}
In conclusion, the boundedness and continuity of the cost $\ell$ imply point~$(iii)$ in Definition~\ref{def:appr_prop}.
\end{proof}
{From Proposition~\ref{prp:prop1} and Lemma~\ref{lema:reach_appr}, we deduce the following result.}
\begin{corollary}\label{corollary}
If the dynamics~\eqref{grugen} have the reachability property at~$x\in\Sigma$, then $\Gamma^{\rm{opt}}[x]$ has the closed graph property.
\end{corollary}
\begin{remark}\label{rmk:k=0}
For $\nu=0$ (i.e., for Euclidean dynamics  $y'=\alpha$) the reachability property is weaker than the assumptions of~\cite{CC} which require that $\Sigma$ is the closure of a bounded open set with $C^2$ boundary. For instance, the set $\Sigma=\{(x_1,x_2)\;:\: x_1\in[0,1], \; x_1^2\leq x_2\leq 2 x_1^2\}$
has the reachability property but does not fulfill the assumptions of~\cite{CC}.
The reachability property may hold even in cases where $\Sigma$ has empty interior; see Remark~\ref{RK4.4} and Example~\ref{example6.2} below.
\end{remark}

\section{A sufficient condition for the reachability property}\label{sectionsufficient}
This section is devoted to establishing the reachability property for each point of~$\Sigma$. If $x$ is a point in the interior of~$\Sigma$, the reachability property at~$x$ always holds true. If $x\in\partial\Sigma$, we need further geometric assumptions near~$x$ to get this property.

We denote by $B_R(x)$ the Euclidean ball centered at ~$x\in\R^2$ with radius~$R>0$.
\begin{definition}\label{def:x1_convex}
We say that a set $A\subset\R^2$ is $x_1$-convex if the following holds:
For any $(x_1,x_2), (x_1+h,x_2)\in A$ with $h\in(0,\infty)$, any point
$(x_1+th,x_2)$, for $t\in[0,1]$, belongs to $A$.
\end{definition}
\begin{remark}
If $(x_1,x_2), (x_1+h,x_2)\in \Sigma$ with $h\in(0,\infty)$ and $\Sigma$ is $x_1$-convex,  the trajectory $(y,\alpha)$ with curve $y(t)=(x_1+th/T, x_2)$ and control $\alpha(t)=(h/T,0) $ is admissible for Grushin dynamics~\eqref{grugen} on $[0,T]$ with $y(0)=(x_1,x_2)$ and $y(T)= (x_1+h,x_2)$.
\end{remark}
 Now we state our assumptions on the set~$\Sigma$ which depend on whether a point $x_0=(x_{01},x_{02})\in\partial\Sigma$ fulfills   $x_{01}=0$ (note that when $x_{01}=0$ the dynamics are degenerate).\\
\underline{Assumptions.} The set~$\Sigma$ satisfies the following:
\begin{itemize}
\item[(H1)]\label{segm}
If $(x_{01},x_{02})\in\partial\Sigma$, with $x_{01}\ne 0$, there exists $R>0$ such that $\Sigma\cap B_R(x_0)$ is $x_1$-convex and there exists a segment, not parallel to the~$x_1$-axis, starting from  $x_0$, contained in $\Sigma$; more precisely,
if $\Sigma\cap B_R(x_0)$ intersects the half-plane $\{y>x_{02}\}$, then there exists $C>0$ such that at least one of the segments
 \[x_1=x_{10},\qquad\forall x_2\in\, (x_{02},x_{02}+R]\]
 or
\[ x_2-x_{02}=C(x_1-x_{01}), \quad \forall x_1\in (x_{01},x_{01}+R]\]
or
\[ x_2-x_{02}=C(-x_1+x_{01}),\quad \forall x_1\in [x_{01}-R,x_{01})\]
is contained in $\Sigma$. We assume similar hypotheses if $\Sigma\cap B_R(x_0)$ intersects the half-plane $\{y<x_{02}\}$.
\end{itemize}
\begin{itemize}
\item[(H2)]\label{curvacar}
If $(0,x_{02})\in\partial\Sigma$, there exists $R>0$ such that $\Sigma\cap B_R(0,x_{02})$ is $x_1$-convex and there exists a suitable curve, starting from  $x_0$, contained in $\Sigma$; more precisely, if $\Sigma\cap B_R(0,x_{02})$ intersects the half-plane $\{y>x_{02}\}$, there exists $C>0$ such that at least one of the curves
\[ x_2-x_{02}=Cx_1^{\nu+1}, \quad\forall x_1\in (0,R]\]
or
\[ x_2-x_{02}=C(-x_1)^{\nu+1},\quad \forall x_1\in [-R,0)\]
is contained in $\Sigma$. An analogous hypothesis is assumed if $\Sigma\cap B_R(0,x_{02})$ intersects the half-plane $\{y<x_{02}\}$.
\end{itemize}

We refer the reader to Section~\ref{sect:ex} for some examples where these assumptions are or are not fulfilled.

\begin{lemma}\label{lemma1gen}
Under assumptions~$(H1)$-$(H2)$ for any $x\in\Sigma$, the control problem~\eqref{grugen}-\eqref{eq:46} has the reachability property. In particular, by Corollary~\ref{corollary}, the set $\Gamma^{\rm{opt}}[x]$ has the closed graph property.
\end{lemma}
\begin{proof}
Fix $x_0\in\Sigma$ and consider a sequence $x_n\in \Sigma$ converging to $x_0$, i.e., $x_n\to x_0$ as in Definition \ref{def:reach_prop}. The strategy of the proof is to  explicitly build the approximating trajectory $(y_n,\alpha_n)\in\Gamma_{0,\delta_n}[x_n]$, $y_n(\delta_n)=x_0$. As a first step (if needed), the approximating trajectory moves parallel to the $x_1$-axis remaining in $\Sigma$, until it reaches a curve as in our assumption.
Afterwards, it follows such a curve until it arrives at the point~$x_0$.
For later use, we recall the definition of the function ``$\sign(\cdot)$'':
\[
\sign(z)=1 \quad\textrm{if } z>0,\qquad  \sign(z)=-1 \quad\textrm{if } z<0,\qquad \sign(0)=0.
\]
Let $R^\prime\leq R$ be such that  for any point~$x^*\in \Sigma\cap B_{R^\prime}(x_{0})$, the $x_1$-horizontal segment from~$x^*$ to the segment of assumption~(H1) (respectively, to the curve of~(H2)) is contained in $\Sigma\cap B_{R}(x_{0})$ and, with   $x^{**}$  denoting the intersection point, the segment from~$x^{**}$ to~$x_0$ (respectively, the arc of the curve in~(H2)) is contained in~$\Sigma\cap B_{R}(x_{0})$.

Observe that it is enough to choose $R^\prime$ verifying
$(\frac{R^\prime}{C})^2+(R^\prime)^2\leq R^2$, with $C$ defined in ~(H1) (respectively,
$(\frac{R^\prime}{C})^\frac{2}{\nu+1}+(R^\prime)^2\leq R^2$, with $C$ defined in ~(H2)).

We assume that $n$ is sufficiently large to have $x_n\in\Sigma\cap B_{R^\prime}(x_{0})$.

Now we construct the trajectory~ $(y_n,\alpha_n)$  connecting $x_n$ to $x_0$ in the interval $[0,\delta_n]$. This construction depends on the position of~$x_0$ according to these three cases:
\begin{enumerate}
\item \label{non caract frontiera} $x_0=(x_{01},x_{02})\in\partial\Sigma$ with $x_{01}\not=0$
\item \label{caract frontiera} $x_0=(0,x_{02})\in\partial\Sigma$
\item \label{interno}$x_0=(x_{01},x_{02})\in\Sigma\backslash \partial\Sigma$.
\end{enumerate}

\noindent \text{{Case \ref{non caract frontiera}}.} First, the trajectory follows the $x_1$-direction from the point $x_n$ to a point on the segment contained in $\Sigma$. Afterward, it follows such a  segment until reaching $x_0$. To this end, we split the arguments according to whether the segment in~$(H1)$ is parallel to the~$x_2$-axis.

If the segment is parallel to the~$x_2$-axis, we choose
\begin{equation*}
\alpha_n(s)=
\left\{\begin{array}{ll}
(\sign\{x_{01}-x_{n1}\},0)&\quad \textrm{ for }s\in[0,s_{1n}]
\\[1mm]
(0,\frac{\sign\{x_{02}-x_{n2}\}}{|x_{01}|^{\nu}})&\quad \textrm{for }s\in[s_{1n},s_{2n}]
\end{array}\right.
\end{equation*}
where $s_{1n}=|x_{01}-x_{n1}|$ and $s_{2n}=s_{1n}+|x_{02}-x_{n2}|=|x_{01}-x_{n1}|+|x_{02}-x_{n2}|$. The corresponding trajectory~$y_n$ satisfies $y_n(\cdot)\in \Sigma$ in $[0,s_{2n}]$ and $y_{n}(s_{2n})=x_{0}$.

If the segment is not parallel to the~$x_2$-axis, without loss of generality, we can assume that the segment in hypothesis~$(H1)$
is
\[
x_2-x_{02}=C(x_1-x_{01}), \,\forall x_1\in \left\{\begin{array}{ll}
&\!\!\!\!\!\!( x_{01},x_{01}+R]\qquad \mbox{if $x_{n2}\geq x_{02}$}\\
&\!\!\!\!\!\![x_{01}-R,x_{01})\qquad \mbox{if $x_{n2}\leq x_{02}$}
\end{array}\right.
\]
for some positive constant~$C$.
Consider the trajectory $(y_n,\alpha_n)$ on $[0,s_{2n}]$ where the control is
\begin{equation*}
\alpha_n(s)\!=\!
\left\{\begin{array}{ll}
(\sign\{\frac{x_{n2}-x_{02}}{C}+x_{01}-x_{n1}\},0)&\  \textrm{for }s\in[0,s_{1n}]\\[1mm]
(\sign\{x_{02}-x_{n2}\},\frac{C\sign\{x_{02}-x_{n2}\}}{|x_{01}+\frac{x_{n2}-x_{02}}{C}+\sign\{x_{02}-x_{n2}\}(s-s_{1n})|^\nu})&\  \textrm{for }s\in[s_{1n},s_{2n}]
\end{array}\right.
\end{equation*}
and $s_{1n}, s_{2n}\in[0,T]$ must satisfy
\begin{equation*}
s_{2n}\geq s_{1n},\qquad y_{n2}(s_{1n})-x_{02}=C(y_{n1}(s_{1n})-x_{01})\qquad\textrm{and}\qquad y_{n}(s_{2n})=x_{0}.
\end{equation*}
Choosing $s_{1n}=|x_{01}-x_{n1}+\frac{x_{n2}-x_{02}}{C}|$ and
$s_{2n}=|x_{01}-x_{n1}+\frac{x_{n2}-x_{02}}{C}|+|\frac{x_{n2}-x_{02}}{C}|$, we get $y_{n}(s_{2n})=x_{0}$ and the curve~$y_n(\cdot)\in\Sigma$ in~$[0,s_{2n}]$.

Moreover, the controls $\alpha_n$ are uniformly bounded:  There exist a constant $C_\nu\in\R^+$ and $N\in\N$  such that
\begin{equation*}
\vert \alpha_n(s) \vert \leq C_\nu(\frac{1}{\vert x_{01}\vert^\nu}+1)\quad \forall n\geq N \quad\forall s\in [0,s_{2n}].
\end{equation*}
Taking $\d_n:=s_{2n}$, we complete the proof in this case.

\vspace*{4pt}\noindent \text{{Case \ref{caract frontiera}}.}
Without loss of generality, we can assume that $x_{n2}>x_{02}$ and that, in hypothesis~$(H2)$, 
the curve contained in $\Sigma$ is
\[ x_2-x_{02}=Cx_1^{\nu+1}, \textit{ for }x_1>0;\]
indeed, the other cases can be dealt with similarly and we omit them.

Consider the trajectory $(y_n,\alpha_n)$ on the time interval $[0,s_{2n}]$ where the control is
\begin{equation*}
\alpha_n(s)=
\left\{\begin{array}{ll}
(-\sign\{x_{n1}-(\frac{x_{n2}-x_{02}}{C})^\frac{1}{\nu+1}\},\ 0)&\quad \textrm{for }s\in[0, s_{1n}]\\[0.5mm]
(-1,\ -C(\nu+1))&\quad \textrm{for }s\in[s_{1n},s_{2n}]\\
\end{array}\right.
\end{equation*}
while $s_{1n}$ and $s_{2n}$ are chosen so that $y_{n1}(s_{1n})>0$, $y_{n2}(s_{1n})-x_{02}=C(y_{n1}(s_{1n}))^{\nu+1}$ and, respectively, $y_{n}(s_{2n})=x_0$.

It is easy to verify that  $s_{1n}=|x_{n1}-\left (\frac{x_{n2}-x_{02}}{C}\right )^\frac{1}{\nu+1}|$, and
$$s_{2n}=\left|x_{n1}-\left (\frac{x_{n2}-x_{02}}{C}\right )^\frac{1}{\nu+1}\right|+\left(\frac{x_{n2}-x_{02}}{C} \right)^\frac{1}{\nu+1}.$$
Moreover, $y_{n1}(s)>0$ for $s\in [s_{1n},s_{2n})$. Hence,  $y_{n}\in \Sigma$ in $[0,s_{2n} ]$  thanks to the $x_1$-convexity and~$(H2)$.
Obviously, $\d_n:=s_{2n}\to 0$ as $n\to \infty$
and we find a uniform bound for the control
$$|\alpha_n(s)|\leq 1+C(\nu+1)\qquad \forall s\in [0,s_{2n}].$$
\vspace*{4pt}\noindent\text{{Case \ref{interno}}.}
Without loss of generality, we can assume that $x_{n2}\geq x_{02}$, ~$x_{01}\geq 0$, and $x_{n1}\geq 0$ because, in the other cases, we can construct similar trajectories joining $x_n$ to $x_0$; see Figure \ref{figura:puntiinterni}.
Let $B_R(x_0)\subset \Sigma$. At least locally, a portion of the curves
\begin{equation}\label{curvaassey}
x_2-x_{02}=C(x_1^{\nu+1}-x_{01}^{\nu+1})\quad \forall x_1>0, \ C>0
\end{equation}
is contained in $B_R(x_0)$ as well as horizontal segments.
 Choose the constant $C>0$ in \eqref{curvaassey} so that the intersection of this curve with $x_2=x_{n2}$ is contained in $B_R(x_0)$.
 Thus, to construct $(y_n,\alpha_n)\in\Gamma_{0,\delta_n}[x_n]$, ~$y_n(\delta_n)=x_0$, remaining in $B_R(x_0)$, we first connect along the $x_1-$direction the point $x_n$ with one point on the curve \eqref{curvaassey} and then follow this curve until reaching $x_0$.

 Using the same arguments as in cases~\ref{non caract frontiera} and~\ref{caract frontiera}, we obtain that $\delta_n\to 0$ and that  $\alpha_n$ are uniformly bounded.
 \end{proof}
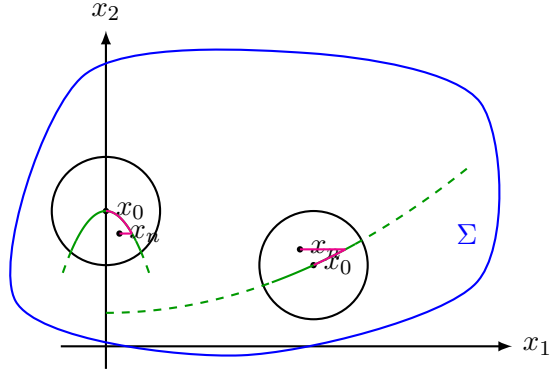
\begin{figure}[h]
\centering
\begin{tikzpicture}[scale=0.6, >=latex]

\draw[->, thick] (-1,0) -- (9,0) node[right]{$x_1$};
\draw[->, thick] (0,-0.5) -- (0,7) node[above]{$x_2$};

\draw[blue, thick]
  plot[smooth cycle] coordinates {
    (-2, 1)
    (-0.5, 6)
    (4.5, 6.5)
    (8.3, 5.4)
    (8.2, 1.4)
    (3, -0.2)
  };
\node[blue] at (8,2.5) {$\Sigma$};
\def\xosx{0}
\def\yosx{3}
\def\rsx{1.2}
\def\aL{1.5}

\draw[black, thick] (\xosx,\yosx) circle (\rsx);
\fill (\xosx,\yosx) circle (2pt) node[right]{$x_0$};

\pgfmathsetmacro\xLs{ -sqrt(\rsx^2/(1+\aL^2)) }
\pgfmathsetmacro\xRs{ +sqrt(\rsx^2/(1+\aL^2)) }

\pgfmathsetmacro\xLsOut{\xLs - 0.1}
\pgfmathsetmacro\xRsOut{\xRs + 0.1}


\draw[green!60!black, thick, dashed, domain=\xLsOut:-1, samples=150]
  plot (\x, {\yosx - \aL*(\x-\xosx)^2});

\draw[green!60!black, thick, domain=\xLs:\xLsOut, samples=150]
  plot (\x, {\yosx - \aL*(\x-\xosx)^2});

\draw[green!60!black, thick, domain=\xLs:\xRs, samples=150]
  plot (\x, {\yosx - \aL*(\x-\xosx)^2});

\draw[green!60!black, thick, domain=\xRs:\xRsOut, samples=150]
  plot (\x, {\yosx - \aL*(\x-\xosx)^2});

\draw[green!60!black, thick, dashed, domain=\xRsOut:1, samples=150]
  plot (\x, {\yosx - \aL*(\x-\xosx)^2});

\def\xn{0.3}
\def\yn{2.5}
\fill (\xn,\yn) circle (2pt) node[right]{$x_n$};

\pgfmathsetmacro\xpL{\xosx + sqrt((\yosx - \yn)/\aL)}
\draw[thick, magenta] (\xn,\yn) -- (\xpL,\yn);
\draw[thick, magenta, domain=\xpL:\xosx, samples=150]
  plot (\x, {\yosx - \aL*(\x-\xosx)^2});

\def\xodx{4.6}
\def\yodx{1.8}
\def\rdx{1.2}

\draw[black, thick] (\xodx,\yodx) circle (\rdx);
\fill (\xodx,\yodx) circle (2pt) node[right]{$x_0$};

\def\aR{0.05}
\pgfmathsetmacro\vR{\yodx - \aR*(\xodx)^2}

\pgfmathsetmacro\xLd{3.676}  
\pgfmathsetmacro\xRd{5.524}

\draw[green!60!black, thick, dashed, domain=0:\xLd, samples=200]
  plot (\x, {\vR + \aR*(\x)^2});
\draw[green!60!black, thick, dashed, domain=\xRd:8, samples=200]
  plot (\x, {\vR + \aR*(\x)^2});

\draw[green!60!black, thick, domain=\xLd:\xRd, samples=200]
  plot (\x, {\vR + \aR*(\x)^2});

\def\xm{4.3}
\def\ym{2.15}
\fill (\xm,\ym) circle (2pt) node[right]{$x_n$};

\pgfmathsetmacro\xpR{ sqrt((\ym - \vR)/\aR) }

\draw[thick, magenta] (\xm,\ym) -- (\xpR,\ym);
\draw[thick, magenta, domain=\xpR:\xodx, samples=200]
  plot (\x, {\vR + \aR*(\x)^2});
  \end{tikzpicture}
  \caption{Reachability at interior points.}\label{figura:puntiinterni}
  \end{figure}

\begin{remark} \label{RK3.2}
In assumption~$(H2)$, the curves $x_2=x_{02}+Cx_1^{\nu+1}$ can be replaced by any curve of the form $x_2=x_{02}+Cx_1^{\rho}$ with $\rho>\nu+1/2$. On the other hand, in condition~$(H1)$ 
the segments can be replaced by curves of the form $x_2=x_{02}+C(x_1^{\rho}-x_{01}^{\rho})$.
\end{remark}
\begin{remark}\label{RK4.4}
A set~$\Sigma$ may fulfill assumptions~$(H1)$ and~$(H2)$ stated above even if it has empty interior; for example, in the standard Grushin case with $\nu=1$, $\Sigma=\{(x_1,x_2)\;\mid\; x_2=x_1|x_1|,\,|x_1|\leq 1\}$.

An example of a set~$\Sigma$, with nonempty interior, fulfilling assumptions~$(H1)$ and~$(H2)$ is one that, at each $(x_1,x_2)\in\partial \Sigma$ with $x_{1}\ne 0$, $\Sigma$ fulfills an interior cone condition while, at each $(0,x_2)\in\partial \Sigma$, it fulfills a ``curved'' interior cone condition modeled on the curves $x_2=x_{02}\pm C(|x_1|)^{\nu+1}$ (see Figure \ref{figura:conostorto}). A classical cone condition at points $(0,x_{2})$ does not ensure the reachability property (see Example \ref{example1}).
\end{remark}

\begin{figure}[h]
\centering
\begin{tikzpicture}[scale=0.9]

\draw[->] (-2,0) -- (2,0) node[right] {$x_1$};
\draw[->] (0,-1) -- (0,2) node[above] {$x_2$};

\draw[thick,blue,domain=0:1,samples=200] plot (\x,{\x*\x});        
\draw[thick,blue,domain=0:1,samples=200] plot (\x,{3*\x*\x}); 

\foreach \x in {0,0.1,...,1}{
  \draw[blue] (\x,{\x*\x}) -- (\x,{3*\x*\x});
}


\end{tikzpicture}
\caption{Curved cone.}\label{figura:conostorto}
\end{figure}
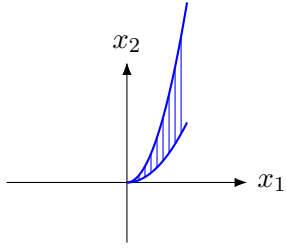

\begin{remark}\label{rmk: costo_traiett}
We can use the argument of the previous proof to connect two points in the interval of time~$[0,t]$ so that~$t$ depends continuously on the distance between them. For instance, we connect $\zeta=(\zeta_1,\zeta_2)$ to $\xi=(\xi_1,\xi_2)$, with $0\leq \zeta_1\leq \xi_1$ and $\zeta_2\geq \xi_2\geq 0$, following first an arc of a curve as in~(H2) and then a segment parallel to the $x_1$-axis provided that the whole curve is contained in~$\Sigma$.
The reverse case, namely, first a segment and then an arc of this curve,  is similar. We can have either a   concave curve or a convex one. For simplicity, we give the details in two cases.

\vspace*{4pt}\noindent\text{Case $1$.} We first consider the control $(1,-C(\nu+1))$ with $C>0$. Hence, we have $y_1(s)=\zeta_1+s$ and $y_2(s)=\zeta_2 -C(y_1(s)^{\nu+1}-\zeta_1^{\nu+1})$.
This curve intersects the line $x_2=\xi_2$ at the point $([\zeta_1^{\nu+1}+(\zeta_2-\xi_2)/C]^{1/(\nu+1)}, \xi_2)$  at time $t_1:=[\zeta_1^{\nu+1}+(\zeta_2-\xi_2)/C]^{1/(\nu+1)}-\zeta_1$. After time~$t_1$, we choose the control $\alpha$ as either $(1,0)$ or $(-1,0)$ according to whether $[\zeta_1^{\nu+1}+(\zeta_2-\xi_2)/C]^{1/(\nu+1)}$ is less or greater than $\xi_1$. In this way, we arrive at point~$\xi$ at time $t=[\zeta_1^{\nu+1}+(\zeta_2-\xi_2)/C]^{1/(\nu+1)}-\zeta_1+|\xi_1-[\zeta_1^{\nu+1}+(\zeta_2-\xi_2)/C]^{1/(\nu+1)}|$.

Using that $[\zeta_1^{\nu+1}+(\zeta_2-\xi_2)/C]^{1/(\nu+1)}\leq \zeta_1+[(\zeta_2-\xi_2)/C]^{1/(\nu+1)}$,
we obtain
$t\leq |\zeta_1-\xi_1|+2[(\zeta_2-\xi_2)/C]^{1/(\nu+1)}$.

\vspace*{4pt}\noindent\text{Case $2$.} If $\zeta_1\ne 0$, we first consider the control $(-1,-C(\nu+1))$ with $C>0$. Hence, we have $y_1(s)=\zeta_1-s$ and $y_2(s)=\zeta_2 +C(y_1(s)^{\nu+1}-\zeta_1^{\nu+1})$; the corresponding curve has the minimum point  at $(0, \zeta_2-C\zeta_1^{\nu+1})$. If $\zeta_2-C\zeta_1^{\nu+1}\leq \xi_2$, we follow the same arguments as in case~$1$ and we arrive at point~$\xi$ at time $t=\zeta_1-[\zeta_1^{\nu+1}-(\zeta_2-\xi_2)/C]^{1/(\nu+1)}+|\xi_1-[\zeta_1^{\nu+1}-(\zeta_2-\xi_2)/C]^{1/(\nu+1)}|$. The time $t$ satisfies
$t\leq |\xi_1-\zeta_1|+2|\zeta_1-[\zeta_1^{\nu+1}-(\zeta_2-\xi_2)/C]^{1/(\nu+1)}|=|\xi_1-\zeta_1|+\frac{2}{\nu+1}\frac{|\xi_2-\zeta_2|}{C\zeta_1^\nu}+o(\frac{|\xi_2-\zeta_2|}{\zeta_1^\nu})$.
\end{remark}

\section{The value function}\label{sect:vf}
In this section, we investigate the continuity of the value function associated with the cost~$J_t$ defined in~\eqref{eq:46}.
Note that, in the case of state constraints, since our dynamics do not satisfy a strong controllability assumption, the continuity of the value function is not a priori guaranteed by the now  classical {\it inward pointing condition}; see~\cite{Soner}.

\begin{definition}
The value function associated with the cost~\eqref{eq:46} and dynamics~\eqref{grugen} is
\begin{equation*}
u(x,t):=\inf\limits_{(y,\alpha)\in\Gamma_t[x]}J_t(x;(y,\alpha)).
\end{equation*}
\end{definition}

In the following lemma, we establish the continuity with respect to time uniformly in space at the final time for the value function.
\begin{lemma}\label{lem:cnc_T}
The value function~$u(x,\cdot)$ is continuous at~$T^-$  uniformly in~$x$, namely,   for every~$\epsilon>0$, there exists $\delta>0$ such that
\begin{equation*}
|u(x,t)-g(x)|\leq \epsilon\qquad \forall t\in(T-\delta,T],\quad \forall x\in\Sigma.
\end{equation*}
\end{lemma}
\begin{proof}
Consider $(x,t)\in\Sigma\times[0,T]$. Since the trajectory starting from~$x$ at time~$t$ with control~$\alpha\equiv 0$ is admissible, we have $u(x,t)\leq \int_t^T \ell(x,\tau)d\tau+g(x)$, namely,
\begin{equation}\label{eq:cnc_T1}
u(x,t)-g(x)\leq \|\ell\|_{L^\infty(\Sigma\times[0,T])} (T-t).
\end{equation}
On the other hand, consider $(y,\alpha)\in\Gamma^{\textrm{opt}}_t[x]$. By standard arguments, we have
\begin{equation*}
u(x,t)\geq \int_t^T \ell(y(\tau),\tau)d \tau +g(y(T))\\
\geq -\|\ell\|_{L^\infty(\Sigma\times[0,T])} (T-t)+o(1)+g(x)
\end{equation*}
where $o(1)\to 0$ as $t\to T^-$ because $y(t)\to y(T)$ and $g$ is continuous.
This last inequality and relation~\eqref{eq:cnc_T1} easily imply the statement.
\end{proof}
For the sake of completeness, in the following lemma we establish that the value function~$u(\cdot,t)$ is LSC. Its proof relies solely on the closed graph property.

\begin{lemma}\label{lemma3}
Assume that the closed graph property holds true for any $(x,t)\in\Sigma\times[0,T)$. Then the function~$u(\cdot,t)$ is LSC for any $t\in[0,T]$.
\end{lemma}

\begin{proof}
Fix $x\in\Sigma$. For $t=T$, the statement stems from the continuity of the function~$g$. Consider $t=0$; the cases for $t\in(0,T)$ are similar so we shall omit them. We want to prove $u(x,0)\leq \liminf_{y\to x}u(y,0)$. Namely, we want to prove that, for any sequence~$\{x_n\}_{n\in\N}$ with $x_n\in\Sigma$ and $x_n\to x$ as $n\to\infty$, there holds
\begin{equation*}
u(x,0)\leq \lim_{n\to\infty}u(x_n,0).
\end{equation*}
To this end, fix such a sequence~$\{x_n\}_{n\in\N}$ and, for each $n\in\N$, consider a trajectory $(y_n,\alpha_n)\in\Gamma^{\textrm{opt}}[x_n]$. Hence, by the same arguments as those in the proof of Proposition~\ref{prp:prop1}, possibly passing to a subsequence, we can assume that there exists a trajectory $(y,\alpha)\in\Gamma[x]$ such that $y_n\to y$ uniformly in $[0,T]$ and $\alpha_n\rightharpoonup \alpha$ weakly in $L^2(0,T)$. The closed graph property ensures that there exists a control law $\alpha^*$ such that $(y,\alpha^*)\in\Gamma^{\textrm{opt}}[x]$. By the lower semicontinuity of the cost~$J$ with respect to the control and the continuity of~$\ell$, we conclude
\begin{equation*}
u(x,0)=J(x;(y,\alpha^*))\leq J(x;(y,\alpha))\leq \lim_{n\to \infty}J(x_n;(y_n,\alpha_n))=\lim_{n\to \infty} u(x_n,0).
\end{equation*}
\end{proof}
In the next proposition, we prove that the assumptions in Proposition~\ref{prp:prop1} are sufficient for the continuity of the value function~$u(\cdot,t)$.
\begin{proposition}\label{prp:cont_x}
Fix $x\in\Sigma$ and assume that either assumption~$(i)$ or assumption~$(ii)$ of Proposition~\ref{prp:prop1} holds true. Then  the value function~$u(\cdot,t)$ is continuous at~$x$, with respect to the first variable, for every $t\in[0,T]$.
\end{proposition}
\begin{proof}For $t=T$, the statement is just the continuity of the function~$g$. Consider $t=0$; the cases for $t\in(0,T)$ are similar so we shall omit them.
By Proposition~\ref{prp:prop1}, $\Gamma^{\rm{opt}}[x]$ has the closed graph property; hence, Lemma~\ref{lemma3} ensures that the function~$u(\cdot,0)$ is LSC at~$x$. In order to complete the proof, we proceed by contradiction, assuming that there exists a sequence $\{x_n\}_n$, with $x_n\to x$ as $n\to\infty$, such that
\begin{equation*}
u(x,0)<\lim_{n}u(x_n,0).
\end{equation*}
We split our arguments according to whether assumption~$(i)$ or assumption~$(ii)$ of Proposition~\ref{prp:prop1} holds true.

Let assumption~$(i)$ hold true. Consider $(y,\alpha)\in\Gamma^{\textrm{opt}}[x]$; the approximation property guarantees that there exists a sequence of trajectories $\{(y_n,\alpha_n)\}_{n\in\N}$ fulfilling $(y_n,\alpha_n)\in\Gamma[x_n]$ and properties~\eqref{eq:lemma1}. Therefore, we have
\begin{equation*}
u(x_n,0)\leq J(x_n;(y_n,\alpha_n)).
\end{equation*}
Passing to the limit, by~\eqref{eq:lemma1}-$(iii)$, we get
\begin{equation*}
\lim_{n\to\infty}u(x_n,0)\leq \lim_{n\to\infty}J(x_n;(y_n,\alpha_n))=J(x;(y,\alpha))=u(x,0)
\end{equation*}
which gives the desired contradiction.

Let assumption~$(ii)$ hold true. The set $\Gamma^{\textrm{opt}}[x]$ reduces to the singleton $\{(x,0)\}$ formed by the trajectory starting from~$x$ with null control; hence, $u(x,0)=\break\int_0^T\ell(x,\tau)\, d\tau+g(x)$. On the other hand, by definition of the value function, we have
\begin{equation*}
u(x_n,0)\leq J(x_n,(x_n,0)) = \int_0^T\ell(x_n,\tau)\, d\tau+g(x_n)
\end{equation*}
where $(x_n,0)$ is the trajectory starting from~$x_n$ with null control. By continuity of~$\ell$ and~$g$, we get $\displaystyle \lim_{n\to\infty}u(x_n,0)\leq u(x,0)$, which gives the desired contradiction.
\end{proof}

In the next statement, we provide a sufficient condition for the local continuity of the value function in both variables $(x,t)$ which amounts to a ``uniform'' reachability property. For global continuity, it is enough to have assumptions~\eqref{cnd:u_cont} with a sufficiently large $\bar r$.

\begin{theorem}\label{thm:u_cont2} Fix $\bar x\in\R^2$. 
Assume that there exist $\bar r >0$ and a modulus of continuity~$\omega$ such that for any $\xi,\zeta\in\Sigma\cap \overline{ B(\bar x,\bar r)}$ there exist $\delta> 0$ and a trajectory $(y,\alpha)\in\Gamma_{0,\delta}[\xi]$ with $y(\delta)=\zeta$ and
\begin{equation}\label{cnd:u_cont}
\delta\leq \omega(|\xi-\zeta|),\qquad
\|\alpha\|_{L^2(0,\delta)}\leq \omega(|\xi-\zeta|).
\end{equation}
Then, the value function~$u$ is continuous in $(\Sigma\cap \overline{B(\bar x,\bar r)})\times[0,T]$.
\end{theorem}
\begin{remark}
Note that in this statement, the two points $\zeta$ and~$\xi$ belong to $\Sigma\cap \overline{ B(\bar x,\bar r)}$ while the curve $y$ must only remain in~$\Sigma$; namely, it can exit $\overline{ B(\bar x,\bar r)}$.
\end{remark}
\begin{proof}
\text{Step $1$.} For $T_1\in[0,T)$, we want to prove that $u$ is continuous in $(\Sigma\cap \overline{B(\bar x,\bar r)})\times[0,T_1]$. To this end, consider $(\xi,t_1),(\zeta,t_2)\in(\Sigma\cap \overline{B(\bar x,\bar r)})\times [0,T_1]$ with $\omega(|\xi-\zeta|)\leq (T-T_1)/2$. By our assumption, there exist a $\delta = \delta(\xi, \zeta)> 0$ and a trajectory $(\tilde y,\tilde \alpha)\in\Gamma_{t_1,t_1+\delta}[\xi]$, with $\tilde y(t_1+\delta)=\zeta$ and
\begin{equation}\label{cnd:u_cont2}
\delta\leq \omega(|\xi-\zeta|)\leq \frac{T-T_1}{2},\qquad
\|\tilde\alpha\|_{L^2(t_1,t_1+\delta)}\leq \omega(|\xi-\zeta|).
\end{equation}
In particular, note that $t_1+\delta\leq (T+T_1)/2<T$.

Fix $(\hat{y},\hat{\alpha})\in\Gamma^{\textrm{opt}}_{t_2}[\zeta]$ and introduce the trajectory $(y,\alpha)\in \Gamma_{t_1}[\xi]$ obtained by the concatenation of $(\tilde y,\tilde \alpha)$ with $(\hat{y},\hat{\alpha})$ rescaled on the time interval $[t_1+\delta, T]$, namely,
\begin{equation*}
\alpha (s)=\left\{\begin{array}{ll}
\tilde \alpha(s)&\quad\textrm{for }s\in[t_1,t_1+\delta]\\
\frac{T-t_2}{T-t_1-\delta}\hat{\alpha}\left(\frac{T-t_2}{T-t_1-\delta}s-\frac{T(\delta+t_1-t_2)}{T-t_1-\delta} \right)&\quad\textrm{for }s\in(t_1+\delta,T].
\end{array}\right.
\end{equation*}
We have
\begin{equation*}
y(s)=\left\{\begin{array}{ll}
\tilde y(s)&\quad\textrm{for }s\in[t_1,t_1+\delta]\\[0.5mm]
\hat{y}\left(\frac{T-t_2}{T-t_1-\delta}s-\frac{T(\delta+t_1-t_2)}{T-t_1-\delta} \right)&\quad\textrm{for }s\in(t_1+\delta,T]
\end{array}\right.
\end{equation*}
and, in particular, $y(T)=\hat{y}(T)$.

By the definition of the value function and by the optimality of~$(\hat{y},\hat{\alpha})$ for $(\zeta,t_2)$, we have
\begin{eqnarray}\notag
u(\xi,t_1)&\!\!\!\!\!\!\leq\!\!\!\!\!\!& J_{t_1}(\xi;(y,\alpha))
=u(\zeta,t_2)+\frac{1}{2}\|\tilde\alpha\|_{L^2(t_1,t_1+\delta)}\\ \notag&& +\frac{1}{2}\frac{t_1-t_2+\delta}{T-t_1-\delta}\|\hat{\alpha}\|^2_{L^2(t_2,T)}+\int_{t_1}^{t_1+\delta}\ell(y(\tau),\tau)d\tau\\ \label{cnd:u_cont3}
&& 
+\int_{t_1+\delta}^{T}\ell(y(\tau),\tau)d\tau-\int_{t_2}^{T}\ell(\hat{y}(\tau),\tau)d\tau.
\end{eqnarray}
We now estimate the several contributions of the last term.
From now on in this proof, we denote by $C$ a constant that may change from line to line and depends only on~$T$ and $T_1$.

Recall that the $L^2$-norms of optimal controls are uniformly bounded; hence, we have
\begin{equation*}
\frac{1}{2}\frac{t_1-t_2+\delta}{T-t_1-\delta}\|\hat{\alpha}\|^2_{L^2(t_2,T)}\leq C\frac{|t_1-t_2|+\delta}{T-t_1-\delta}
\leq C(|t_1-t_2|+\delta).
\end{equation*}
By the boundedness of the function~$\ell$, we infer
\begin{equation*}
\int_{t_1}^{t_1+\delta}\ell(y(\tau),\tau)d\tau\leq C \delta.
\end{equation*}
Furthermore, we have
\begin{align*}
&\int_{t_1+\delta}^{T}\ell(y(\tau),\tau)d\tau-\int_{t_2}^{T}\ell(\hat{y}(\tau),\tau)d\tau\\
&=\frac{T-t_1-\delta}{T-t_2}\int_{t_2}^{T}\ell\left(\hat{y}(\tau), \frac{\tau(T-t_1-\delta)+T(\delta+t_1-t_2)}{T-t_2}\right)d\tau -\int_{t_2}^{T}\ell(\hat{y}(\tau),\tau)d\tau\\
&=\frac{T-t_1-\delta}{T-t_2}\left[\int_{t_2}^{T}\ell\left(\hat{y}(\tau), \frac{\tau(T-t_1-\delta)+T(\delta+t_1-t_2)}{T-t_2}\right)d\tau-\int_{t_2}^{T}\ell(\hat{y}(\tau),\tau)d\tau\right]\\
&\quad+\left[\frac{T-t_1-\delta}{T-t_2}-1\right]\int_{t_2}^{T}\ell(\hat{y}(\tau),\tau)d\tau\\
&\leq C\omega_\ell\left(C\left(|t_1-t_2|+\delta\right)\right)+C\left(|t_1-t_2|+\delta\right)
\end{align*}
where $\omega_\ell$ is the modulus of continuity of the function~$\ell$ in the set~$\Sigma\times[0,T]$.

Substituting all these estimates into inequality~\eqref{cnd:u_cont3} and taking into account~\eqref{cnd:u_cont2}, we obtain
\begin{equation*}
\begin{array}{rcl}
u(\xi,t_1)-u(\zeta,t_2)&\!\!\!\!\leq\!\!\!\!& \omega(|\xi-\zeta|)+C(|t_1-t_2|+\delta)\\&&+C \delta+C\omega_\ell\left(C\left(|t_1-t_2|+\delta\right)\right)+C\left(|t_1-t_2|+\delta\right).
\end{array}
\end{equation*}
By relation~\eqref{cnd:u_cont2}, we have $\delta\leq \omega(|\xi-\zeta|)$; hence, the previous inequality entails that there exists a modulus of continuity~$\tilde \omega$ such that
\begin{equation*}
u(\xi,t_1)-u(\zeta,t_2)\leq \tilde \omega(|\xi-\zeta|+|t_1-t_2|).
\end{equation*}
Reversing the roles of $(\xi,t_1)$ and $(\zeta,t_2)$, we complete the proof of Step~$1$.

\vspace*{4pt}\noindent\text{Step $2$.} Taking into account Step~$1$, it remains to prove the continuity in $(\Sigma\cap \overline{B(\bar x,\bar r)})\times[T_1,T]$ where $T_1$ is arbitrary. To this end, for $\xi,\zeta \in \Sigma\cap \overline{B(\bar x,\bar r)}$, we have
\begin{equation*}
|u(\xi,t_1)-u(\zeta,t_2)|\leq |u(\xi,t_1)-g(\xi)|+|u(\zeta,t_2)-g(\zeta)|+|g(\xi)-g(\zeta)|.
\end{equation*}
Lemma~\ref{lem:cnc_T} ensures that for every $\epsilon>0$, there exists $T_1\in[0,T)$ such that $|u(x,t)-g(x)|\leq \epsilon$ for every $(x,t)\in (\Sigma\cap \overline{B(\bar x,\bar r)})\times(T_1,T)$.
Choosing $t_1,t_2>T_1$ and $\xi$ and $\zeta$ such that $|g(\xi)-g(\zeta)|<\epsilon$, we establish the result.
\end{proof}
\begin{remark}
In Section \ref{sect:ex} we provide two simple examples where relations \eqref{cnd:u_cont} are satisfied and then
the value function is continuous; see Examples \ref{example5.1} and
\ref{ex:5.8}.
\end{remark}

\section{Examples and counterexamples}\label{sect:ex}
In this section, we collect several examples of control problems where the closed graph property and/or the reachability property hold or do not hold. We also discuss the continuity of the value function.

\begin{example}\label{example5.1}
By Lemma~\ref{lemma1gen}, a first case where the control problem has the reachability property is when
$\Sigma=[a_1, b_1] \times [a_2, b_2]$. By Remark~\ref{rmk: costo_traiett}, one can easily obtain the relations in~\eqref{cnd:u_cont} and consequently the continuity of the  value function.
\end{example}

\begin{example} \label{exh(x,y)} Consider   an $x_1$-convex, compact set $\Sigma$ with a $C^k$ boundary with $k=\lfloor\nu +\frac{1}{2}\rfloor+1$.
To explain this local assumption, for instance, for $\bar x=(0,0)$, we assume that there exists $r>0$ such that $\Sigma\cap B_r(\bar x)=\{x\in B_r(\bar x)\,:\,\; f(x_1)-x_2\leq 0\}$ with $f\in C^k(-r,r)$.
We suppose that $f^{(j)}(0)=0$ for $j=0,\dots, (k-1)$.
Then  the assumptions of Lemma \ref{lemma1gen} (or more generally of Remark \ref{RK3.2}) are satisfied and
the reachability property holds true at~$\bar x$.
\noindent For instance, for $\nu=1$, one can prove that the curve $x_2=C\, x_1^\rho$ with $\rho\geq 2$ is locally contained in $\Sigma$ for suitable~$C$ and $\rho$, by showing that $g(x_1)=f(x_1)-C\, x_1^\rho\leq 0$ in a neighborhood of~$\bar x$ (see Figure \ref{figura:new5.2}).
\end{example}
\begin{figure}[h]
\centering
\begin{tikzpicture}[scale=0.7][x=1cm,y=1cm,>=Stealth]


\def\R{3}

\def\A{0.1397542486}

\def\C{0.60}

\def\Y{2.2360679775}

\def\H{3.15}

\def\W{2.70}

\def\B{0.35}

\def\L{0.9139320225}



\foreach \n in {0,...,447}{

    \pgfmathsetmacro{\yy}{\n*\Y/447}

    \pgfmathsetmacro{\yyb}{(\n+1)*\Y/447}

    \pgfmathsetmacro{\xx}{
        (\yy/\A)^(1/4)
    }

    \pgfmathsetmacro{\eps}{0.002}

    \fill[blue!10]
    (-\xx,\yy)
    rectangle
    (\xx,\yyb+\eps);
}


\foreach \n in {0,...,364}{

    \pgfmathsetmacro{\t}{\n/364}

    \pgfmathsetmacro{\yy}{
        \Y+\L*\t
    }

    \pgfmathsetmacro{\yyb}{
        \Y+\L*(\t+1/364)
    }

    \pgfmathsetmacro{\xx}{
        2
        +0.2043614129*\t
        +1.6912771742*\t*\t
        -1.1956385871*\t*\t*\t
    }

    \pgfmathsetmacro{\eps}{0.002}

    \fill[blue!10]
    (-\xx,\yy)
    rectangle
    (\xx,\yyb+\eps);
}



\foreach \n in {0,...,139}{

    \pgfmathsetmacro{\yy}{
        \H+\n*\B/139
    }

    \pgfmathsetmacro{\yyb}{
        \H+(\n+1)*\B/139
    }

    \pgfmathsetmacro{\xx}{
        \W*sqrt(
            max(
                0,
                1-((\yy-\H)/\B)^2
            )
        )
    }

    \pgfmathsetmacro{\eps}{0.002}

    \fill[blue!10]
    (-\xx,\yy)
    rectangle
    (\xx,\yyb+\eps);
}


\draw[
    green!60!black,
    very thick
]
(0,0) circle (\R);


\draw[->,thick]
(-3.5,0) -- (3.5,0);

\draw
(3.5,0) node[right] {$x_1$};

\draw[->,thick]
(0,-0.4) -- (0,3.7);

\draw
(0,3.7) node[above] {$x_2$};


\draw[
    blue!70!black,
    very thick
]
plot[
    domain=-2:2,
    samples=200
]
(
    \x,
    {\A*\x*\x*\x*\x}
);


\draw[
    blue!70!black,
    very thick
]
plot[
    domain=0:1,
    samples=200
]
(
    {
        2
        +0.2043614129*\x
        +1.6912771742*\x*\x
        -1.1956385871*\x*\x*\x
    },
    {
        \Y+\L*\x
    }
);


\draw[
    blue!70!black,
    very thick
]
plot[
    domain=0:1,
    samples=200
]
(
    {
        -2
        -0.2043614129*\x
        -1.6912771742*\x*\x
        +1.1956385871*\x*\x*\x
    },
    {
        \Y+\L*\x
    }
);


\draw[
    blue!70!black,
    very thick
]
plot[
    domain=0:180,
    samples=200
]
(
    {\W*cos(\x)},
    {\H+\B*sin(\x)}
);


\draw[
    magenta,
    very thick
]
plot[
    domain=-1.85:1.85,
    samples=200
]
(
    \x,
    {\C*\x*\x}
);


\draw
(3.1,1.10)
node[
    blue!70!black,
    font=\large
]
{$x_2=f(x_1)$};

\draw[
    ->,
    blue!70!black,
    thick
]
(2.30,0.9)
-- (1.35,0.48);


\draw
(-3.9,1.6)
node[
    magenta,
    font=\large
]
{$x_2=Cx_1^2$};

\draw[
    ->,
    magenta,
    thick
]
(-2.9,1.30)
-- (-1.15,0.66);


\draw
(0.85,2.55)
node[
    font=\large
]
{$\Sigma$};

\end{tikzpicture}
\caption{Example \ref{exh(x,y)} for $\nu=1$.}\label{figura:new5.2}
\end{figure}

In the following example, we provide a case where  the uniform reachability \eqref{cnd:u_cont} is fulfilled; in particular, by Corollary~\ref{corollary} and by Theorem~\ref{thm:u_cont2}, the closed graph property and   continuity of the value function hold true.

\begin{example}\label{ex:5.8}
Consider $\Sigma=\{(x_1, x_2)\in \re^2 : 0\leq x_1\leq 1, x_1^{\nu+1}\leq x_2\leq 1\}$. Let $z_1$ and $z_2$ be two points of $\Sigma$. We construct $(y,\alpha)\in \Gamma_{0,\delta}[z_1]$, such that $y(\delta)=z_2$ where $\delta$ is a modulus of continuity of $|z_2-z_1|$ and $\alpha$ is uniformly bounded independently of $z_1$ and $z_2$. This can be done by the concatenation of a trajectory parallel to the $x_1$-axis (where $\alpha=(\pm 1,0)$) and a trajectory where $\alpha=(-1,-(\nu+1))$ (see Figure~\ref{figura:ex6.6}). By Remark \ref{rmk: costo_traiett}, we get the relations in~\eqref{cnd:u_cont}.
\end{example}

\begin{figure}[h]
\centering
\begin{tikzpicture}[scale=3, >=latex]

\draw[->, thick] (-0.1,0) -- (1.2,0) node[right] {$x_1$};
\draw[->, thick] (0,-0.1) -- (0,1.2) node[above] {$x_2$};

\draw[very thick, blue] (0,1) -- (1,1); 
\draw[very thick, blue, domain=0:1, samples=200] plot (\x, {\x*\x}); 
\draw[very thick, blue] (0,0) -- (0,1); 
\node[blue] at (1.4,0.2) {\small $\Sigma= \{0 \le x_1 \le 1,\; x_1^{\nu+1} \le x_2 \le 1\}$};

\fill[pattern=north east lines, pattern color=blue!40, domain=0:1, variable=\x]
  (0,1)
  -- plot ({\x}, {\x*\x})
  -- (1,1)
  -- cycle;
\filldraw[magenta] (1/3,1) circle (0.015) node[above left] {$z_1$};
\filldraw[magenta] (1/2,1) circle (0.015) node[above right] {$z_2$};
\draw[magenta, very thick]
  (1/3,1) -- (1/2,1);

\filldraw[magenta] (0,0.6) circle (0.015) node[left] {$z_1$};
\filldraw[magenta] (0,0.5) circle (0.015) node[left] {$z_2$};

\draw[magenta, very thick] (0,0.6) -- (0.316,0.6); 
\draw[magenta, very thick, domain=0.316:0, samples=200]
plot (\x, {\x*\x + 0.5});

\filldraw[magenta] (0.25,0.6667) circle (0.015) node[above left] {$z_1$};
\filldraw[magenta] (0,0) circle (0.015) node[below right] {$z_2$};

\draw[magenta, very thick] (0.25,0.6667) -- (0.816,0.6667); 
\draw[magenta, very thick, domain=0:0.816, samples=200]
  plot (\x, {\x*\x});
\end{tikzpicture}
\caption{Example \ref{ex:5.8}.}\label{figura:ex6.6}
\end{figure}

In the following example we show a case where  $(1)$ each point in $\Sigma$ can be reached from any other point by an admissible trajectory, $(2)$ the reachability property does not hold, and $(3)$ the closed graph property holds.

\begin{example}\label{example1}
Fix $0<m_1<m_2$ and consider the set
\[\Sigma=\{(x_1,x_2)\in \R^2\;:\; x_1\in [0,1],\,m_1 x_1\leq x_2\leq m_2 x_1\}
\]
with Grushin dynamics and $\nu=1$.\\
$(1)$. We claim that any point $x\in\Sigma$ can be reached from any other point $x_0=(x_{01},x_{02})\in \Sigma$ following an admissible trajectory. Indeed, consider $x=0$ and any $x_0\in\Sigma$, $x\ne x_0$; we shall omit the cases where~$x\ne0$ because they are simpler. The trajectory $(y,\alpha)$, with $y(0)=x_0$ and $\alpha(s)=(-1, -\frac{x_{02}}{x_{01}(x_{01}-s)})$, fulfills dynamics~\eqref{grugen} with $y(x_{01})=0$ and its support is the straight segment from $x_{0}$ to $0$.\\
$(2)$. Now, we want to establish that the point $x=0$ is unreachable; namely, we want to prove that, for any $x_0\in\Sigma$, any admissible trajectory $(y,\alpha)$ with $y(0)=x_0$ and $y(T)=0$ has infinite cost.
Indeed, for $x_0=(x_{01},x_{02})\in\Sigma$ fixed, consider any admissible trajectory $(y,\alpha)$ with $y(0)=x_0$ and $y(T)=0$. Since $y(s)\in\Sigma$ for any $s\in[0,T]$,   there holds 
\begin{equation*}
0\leq m_1 y_1(s)\leq y_2(s)\leq m_2 y_1(s)\qquad\forall s\in[0,T];
\end{equation*}
in particular, we deduce $y_2'(s)\geq -\frac{|\alpha_2(s)|}{m_1}y_2(s)$ and, by Gronwall' lemma, 
\begin{equation*}
y_2(s)\geq x_{02}\exp\left\{-\frac{1}{m_1}\int_0^s |\alpha_2(\tau)|\, d\tau \right\}.
\end{equation*}
Then the condition $y_2(T)=0$ implies $\int_0^T |\alpha_2(\tau)|\, d\tau =\infty$ and, in particular, that $(y,\alpha)$ has infinite cost.\\
$(3)$. Let us now prove that the closed graph property holds. For any $\Sigma\ni x\ne 0$, the reachability property is obvious so, from Corollary~\ref{corollary}, we get the result. For $x=0$, the closed graph property stems from point~$(2)$ and Proposition~\ref{prp:prop1}-$(ii)$.
\end{example}

\begin{figure}[h]
\centering
\begin{tikzpicture}[scale=0.8]

\draw[->] (-1,0) -- (6,0) node[right] {$x_1$};
\draw[->] (0,-1) -- (0,3) node[above] {$x_2$};

 \draw[thick, fill=blue!10] (0,0) -- (4,2) -- (2,4) -- cycle;

\node[blue] at (2,1.5) {$\Sigma$};

\node[blue] at (4.2,1.3) {$x_2=m_1x_1$};

\node[blue] at (1.5,2.8) {$x_2=m_2x_1$};

\end{tikzpicture}
\label{figuraex5.4}
\caption{Example \ref{example1}.}
\end{figure}

In the following example, the closed graph property holds for $x=0$ both in the case where the reachability property holds and in the case where the reachability property does not hold.

\begin{example}\label{example6.2}
Consider the case where $\Sigma$ coincides with a regular curve: $\Sigma=\{(x_1,x_2)\in\R^2\;:\; x_2=\gamma(x_1)\;,\; x_1\in [0,R]\}$, for some $R>0$, where $\gamma$ is a regular function with $\gamma(0)=0$ (namely, $0\in \Sigma$). We want to prove that the closed graph property holds for~$x=0$.

Consider any point $0\ne x_0=(x_{01}, x_{02})\in\Sigma$. We have two possibilities: either there exists $(y,\alpha)\in \Gamma[x_0]$ and $y(T)=0$ or such a trajectory does not exist. We study these two cases separately.

\vspace*{4pt}\noindent \text{Case $1$. $\exists (y,\alpha)\in \Gamma[x_0]$ with $y(T)=0$.} Since $\Sigma$ coincides with a curve, the support of ~$y$ must contain the portion of~$\Sigma$ with $0\leq x_1\leq x_{01}$.
Now, if we prove that the reachability property for $x=0$ holds true, then the closed graph property follows by Corollary~\ref{corollary}.

In order to prove the reachability property, since the curve ~$y$ must pass through every point $\bar x=(\bar x_{1},\gamma(\bar x_{1}))$ with $0\leq \bar x_{1}\leq x_{01}$, we define $\bar s\in(0,T)$ as the last instant when $y(\cdot)$ coincides with $\bar x$.

Consider a sequence $\{\bar x^n\}$, $\bar x^n=(\bar x_{1}^n,\gamma(\bar x_{1}^n))$, with $\bar x^n\in \Sigma$ and $\bar x_{1}^n\to 0$ as $n\to \infty$. For each $n\in\N$, let $\bar s^n$ be the last time when the curve  ~$y(\cdot)$ passes through $\bar x^n$. Since $y\in C^{1/2}([0,T])$,   $\lim_{n\to\infty}\bar s^n=T$.
The family of trajectories $(y_n,\alpha_n)$ with $y_n(0)=\bar x^n$ and $\alpha_n(s)= \alpha (s+\bar s^n)$ fulfills the requirements for the reachability property with $\delta_n=T-\bar s^n$.

\vspace*{4pt}\noindent \text{Case $2$. $\nexists (y,\alpha)\in \Gamma[x_0]$ with finite cost and $y(T)=0$.} The closed graph property is a straightforward consequence of Proposition~\ref{prp:prop1}-$(ii)$.
\end{example}

In the following example we find a set where neither the reachability property nor the closed graph property  holds for $x=0$. Furthermore, the value function may be discontinuous.
\begin{example}\label{ex5.6}
Denote by $\Sigma_1$ the set defined in Example~\ref{example1} (see Figure \ref{figuraex5.4}). Consider the set $\Sigma=\Sigma_1 \cup \{(x_1, x_2)\in \re^2 : -1\leq x_2\leq 0, -1\leq x_1\leq 1\}$ (see Figure~6), a sequence of points $\{x_n\}_{n\in \N}$, with $x_n\in\Sigma_1$ and $x_n\to 0$, and a sequence of optimal trajectories $(y_n,\alpha_n)\in\Gamma^{\rm{opt}}[x_n]$ with $y_n$ uniformly convergent to some~$y$.
As in Example~\ref{example1}, the trajectory~$(y,\alpha)$ fulfills $y(\cdot)=0$ and $\alpha(\cdot)=0$ on~$[0,T]$. Here, the trajectory $(0,0)$ may not be optimal for $x=0$, i.e., the closed graph property does not hold. Indeed, an optimal trajectory starting from $x=0$ can move in
the set 
if, for instance, the running cost~$\ell$ is strictly decreasing with respect to $x_1$ and $T$ is large enough.

Note that the set $\Sigma$ is the closure of an open set and that, in the controllable case (namely, $\nu=0$), it  enjoys the reachability property and, consequently, the closed graph property.
\end{example}

\begin{figure}[h]
\centering
\begin{tikzpicture}[>=Latex,scale=0.4]

  \draw[->,thick] (-5,0)--(6,0) node[right] {$x_{1}$};
  \draw[->,thick] (0,-2)--(0,5) node[above] {$x_{2}$};

   \draw[thick, fill=blue!10] (-4,-2) rectangle (4,0);
 \draw[thick, fill=blue!10] (0,0) -- (4,2) -- (2,4) -- cycle;


  \node[blue] at (3,1) {$\Sigma$};
\end{tikzpicture}
\label{figuraex5.6}
\caption{Example \ref{ex5.6}.}
\end{figure}

\begin{remark}
We refer the reader to \cite[Example IV.5.3]{BCD} for an example with different, very degenerate, dynamics where it is not true that any point can be reached from any other point by an admissible trajectory and where the closed graph property does not hold. It is worth noting that, if in this example the dynamics are replaced by Grushin dynamics, then the closed graph property holds true.
\end{remark}

\begin{remark}\label{rem_char}
Let $\Sigma$ be the closure of a regular open set, i.e., $\Sigma=\{x\in \R^2\,:\, h(x_1,x_2)\leq 0\}$ with $h\in C^2$, $\nabla h(x_1,x_2)\ne 0$, and $\Sigma$ $x_1-$convex.

Let us first consider the standard Grushin dynamics with $\nu=1$. From Lemma \ref{lemma1gen} and Remark \ref{RK3.2}, the reachability property holds true.

 Indeed, let us focus on the so-called {\it characteristic points} of $\partial\Sigma$ which are the points where the  Grushin horizontal gradient of $h$ vanishes (that is $<X_i(x_0),\nabla h(x_0)\break>=0$ for $i=1,2$). As one can easily check, the characteristic points are the points $x_0=(0,x_{02})$ with $\partial_{x_1}h(x_0)=0$ and, at these points, $\partial\Sigma$ is locally a graph of an $x_1-$function $f$ with $f^\prime(0)=0$. Thus, the reachability property follows as in Example~$5.2$.

Observe that at these points~$x_0$, the normal vector $n(x_0)$ to $\partial\Sigma$ is in the $x_2$-direction so that for any choice of the control $\alpha$, the standard Grushin dynamics are orthogonal to  $n(x_0)$. Hence,
the classical inward/outward pointing condition fails (see, e.g., \cite{BFZ} and the references therein).

More generally, if $\nu>1$, in line with Example \ref{exh(x,y)}, in order to obtain the reachability property for $\Sigma$, we assume that $\Sigma$ has $C^k$ boundary with $k=\lfloor\nu +\frac{1}{2}\rfloor+1$ and that, at the characteristic points, the function $f$ which locally defines the boundary of $\Sigma$ near $x_0=(0,x_{02})$ satisfies  $f^{(j)}(0)=0$ for $j=0,\dots, (k-1)$.
\end{remark}

\section{A class of mean field games}\label{sect:MFGequil}

We are now ready to tackle mean field games where the running  and terminal costs depend on the distribution of states. Following the arguments of~\cite{24SIAM,CC} and using the results contained in previous sections, we prove the existence of a mild solution.

\subsection{Setting and notation}\label{subsec:setting}

\underline{Probability sets}. Let $\cP(\Sigma)$ denote the set of Borel probability measures on $\Sigma$ endowed with the narrow topology. Similarly, $\cP(\Gamma)$ stands for the set of Borel probability measures on $\Gamma$.  Let $\cP_1(\Sigma)$ be the set of measures in $\cP(\Sigma)$ with finite first moment.\\
\underline{Assumptions.}
In the rest of this section we shall assume the following hypotheses.
\begin{itemize}
\item[(i)] The running cost is a map $L\in C(\cP(\Sigma);C_{b}(\Sigma\times [0,T]))$ while the terminal cost is a map $G\in C(\cP(\Sigma); C_{b}(\Sigma))$.
The images of $m\in \cP(\Sigma)$ by $L$, respectively by $G$, are denoted by $L[m](\cdot,\cdot)$, respectively $G[m](\cdot)$.
Furthermore, there exists  $K$ defined by
\begin{equation}\label{eq:37}
K=\max\left(\sup_{m\in \cP(\Sigma)} \|L[m]\|_{L^\infty} , \sup_{m\in \cP(\Sigma)} \|G[m]\|_{L^\infty}\right)\in\R^+.
\end{equation}
\item[(ii)]
For each point $x\in\Sigma$, one of the two following assumptions holds true:
\begin{itemize}
\item[($a$)] the Grushin-type dynamics~\eqref{grugen} fulfill \eqref{eq:lemma1}-($i$) and -($ii$)
\item[($b$)] for any $\bar x\in \Sigma\setminus\{x\}$, there is no admissible trajectory $(y,\alpha)\in W^{1,2}([0,T])\times L^2(0,T)$ with $y(0)=\bar x$ and $y(T)=x$.
\end{itemize}
\end{itemize}
\underline{The set $\Gamma_C$.}
Let us introduce for $C>0$,
\begin{align*}
\Gamma_C[x]&=\left\{y\in W^{1,2}([0,T]):\exists \alpha\in L^2(0,T)\textrm{ s.t. } (y,\alpha)\textrm{ is adm. for $x$},
\  \| \alpha \|_2\leq C\right\},
\\\Gamma_C&=\bigcup_{x\in\Sigma}\Gamma_C[x].
\end{align*}
The set~$\Gamma_C$, endowed with the topology of uniform convergence, is compact (see  \cite[Lemma 3.1]{24SIAM}).

We recall  that, given $y\in\Gamma_C[x]$, the control~$\alpha$ is not uniquely determined, as remarked before in Section \ref{sect:settings}.\\
\underline{The set $\cP(\Gamma_C)$ and the associated costs.}
Let $\cP(\Gamma_C)$ denote the set of probability measures on~$\Gamma_C$ endowed with the narrow topology.
For any $\mu\in \cP(\Gamma_C)$ and  $t\in [0,T]$,  $m^\mu(t)\in \cP( \Sigma)$ is defined by $m^\mu(t)=e_t\sharp \mu$,  where  $e_t:\Gamma_C\to \Sigma$ is  the evaluation map  defined by $e_t(y)=y(t)$. For any $\mu\in \cP(\Gamma_C)$, $\textrm{supp}(m^\mu(t))\subset\{x\in\Sigma\;:\;|x|\leq C\}$, and the map $t\mapsto m^\mu(t)$ belongs to $C^{1/2}([0,T];\cP_1(\Sigma))$ (see \cite[Lemma 3.2]{22SIAM}). Hence, for all $y\in W^{1,2}([0,T])$, the function~$t\mapsto L[m^\mu(t)](y(t),t)$ is continuous and bounded by the constant $K$ in (\ref{eq:37}).

With any $\mu\in \cP(\Gamma_C)$ and any admissible trajectory $(y,\alpha)\in W^{1,2}([0,T])\times L^2(0,T)$, we associate the cost
\begin{equation}\label{costMFG}
J^\mu(x;(y,\alpha))
=\int_0^T \left(L[m^\mu(\tau)] (y(\tau),\tau)+\frac{|\alpha(\tau)|^2}{2} \right)d\tau+  G[ m^\mu(T)](y(T)).
\end{equation}
\begin{remark}
By the arguments of Remark~\ref{rmk:2.1}, assumption~$(ii)-(a)$ implies that the control problem \eqref{grugen}-\eqref{costMFG} with any $\mu\in \cP(\Gamma_C)$  fulfills the approximation property at~$x$. On the other hand, by the arguments of Remark~\ref{rmk:unreach_ell}, assumption~$(ii)-(b)$ implies that the point $x$ is unreachable for the control problem \eqref{grugen}-\eqref{costMFG} with any $\mu\in \cP(\Gamma_C)$.
\end{remark}

\noindent\underline{Optimal trajectories.} Fix $\mu\in \cP(\Gamma_C)$. For any $x\in \Sigma$, let us set
\begin{equation}\label{eq:42}
\Gamma^{\mu,{\rm opt}}[x]=\left\{(y,\alpha)\in \Gamma[x]\;:\; J^\mu(x;(y,\alpha))=\min_{ (\widetilde y,\widetilde \alpha) \in \Gamma^\mu[x]}  J^\mu(x; (\widetilde y,\widetilde \alpha)) \right\}
\end{equation}
where $J^\mu $ is defined in (\ref{costMFG}) and $\Gamma^\mu[x]$ is the set of admissible trajectories starting from~$x$ with finite cost~$J^\mu$.
Remark~\ref{prp:ex_OT} entails that for each $\mu\in \cP(\Gamma_C)$ and $x\in \Sigma$, the set $\Gamma^{\mu,{\rm opt}}[x]$  is not empty. We set $\Gamma^{\mu,{\rm opt}}=\cup_{x\in\Sigma} \Gamma^{\mu,{\rm opt}}[x]$.

\begin{remark}\label{rmk:OT_unif}

From assumption~\eqref{eq:37}, there exists a positive constant~$\tilde C$ such that, for every~$\mu\in \cP(\Gamma_C)$, $x\in\Sigma$, and $(y,\alpha)\in \Gamma^{\mu,{\rm opt}}[x]$, there holds $\|\alpha\|_2\leq \tilde C$.
Indeed, assumption~\eqref{eq:37} entails that any $(y,\alpha)\in\Gamma^{\mu,{\rm opt}}[x]$ fulfills $\|\alpha\|_{L^2}\leq\sqrt{ 2K(T+1)}$  (just using the trajectory~$(x,0)$ with null control as a competitor).

In particular, for every $\mu\in \cP(\Gamma_C)$, $x\in \textrm{supp}(m_0)$, and $(y,\alpha)\in \Gamma^{\mu,{\rm opt}}[x]$, there holds $y\in\Gamma_{\tilde C}[x]$ (possibly after taking a larger value of $\tilde C$).
\end{remark}
\begin{remark}\label{sigma'}
Assume that $\Sigma$ is closed but not bounded and that $m_0$ has compact support. Then assumption \eqref{eq:37} allows us to achieve all the previous results. Indeed, it is enough to restrict our analysis  to the compact set $\Sigma^\prime=(supp(m_0)+\overline{B_R})\cap \Sigma$ where $R=R(K)$ and $K$ is the constant defined in \eqref{eq:37}, provided that $\Sigma^\prime$ satisfies the assumptions of the previous sections in order to guarantee the closed graph property.
\end{remark}

\noindent\underline{The set   $\cP_{m_0}(\Gamma_C)$.}
Let $\cP_{m_0}(\Gamma_C)$ denote the set of measures $\mu\in\cP(\Gamma_C)$ such that $e_0\sharp \mu=m_0$.
For $C$ sufficiently large, $\cP_{m_0}(\Gamma_C)$ is not empty (see \cite[Lemma~3.3]{24SIAM}).

\subsection{Relaxed equilibria}
Fix $\mu\in \cP_{m_0}(\Gamma_C)$ and $x\in \Sigma$. For $J^\mu$ defined in (\ref{costMFG}), let us set
\begin{equation}\label{eq:42new}
\GammamuoptC[x]=\left\{y\in \Gamma_C[x]\;:\;\exists \alpha\in L^2(0,T)\textrm{ s.t. } (y,\alpha)\in\Gamma^{\mu,{\rm opt}}[x]\right\}.
\end{equation}

\begin{definition}
\label{sec:setting-notation-3}
The complete probability measure $\mu \in \cP_{m_0}(\Gamma_C)$ is a relaxed equilibrium associated with the initial distribution $m_0$ if
\begin{equation*}\label{eq:43}
{\rm{supp}}(\mu)\subset \mathop \bigcup_{
x\in {\rm{supp}}(m_0)}\GammamuoptC[x].
\end{equation*}
\end{definition}
We can now state our main result.
\begin{theorem}\label{sec:setting-notation-6}
Under the assumptions stated in Section~\ref{subsec:setting}, for $C$ large enough, there exists a relaxed equilibrium $\mu\in \cP_{m_0}(\Gamma_C)$.
\end{theorem}
 In order to prove Theorem~\ref{sec:setting-notation-6}, we need a closed graph property. The proof of Theorem~\ref{sec:setting-notation-6} is postponed to the end of this section.
 \begin{lemma}\label{prop2} Consider $C\geq \tilde C$ where $\tilde C$ is the constant introduced in Remark~\ref{rmk:OT_unif} and such that $\cP_{m_0}(\Gamma_C)\not= \emptyset$. Fix $\mu\in\cP_{m_0}(\Gamma_C)$ and $x\in\textrm{supp}(m_0)$. Let $(\mu_n)_{n\in\N}$ be a sequence of probability measures with $\mu_n\in\cP_{m_0}(\Gamma_C)$, narrowly convergent to~$\mu$ as $n\to\infty$,
  and  let $(x_n)_{n\in\N}$ be a sequence of points    
   with $x_n\in\Sigma$ and $x_n\to x$ as $n\to\infty$.
Let~$(y_n)_{n\in\N}$ be  such that $y_n\in \Gamma_C^{\mu_n,{\rm opt}}[x_n]$ and $y_n$  converges uniformly to some $y$ as $n\to\infty$.
Then  $y\in \GammamuoptC[x]$.
\end{lemma}
\begin{proof} The arguments of this proof are similar to those of~Proposition \ref{prp:prop1} and \cite[Lemma 3.4]{CC} and  depend on  the structure of the cost functional. The continuity of the running cost $L$ with respect to $\mu$ plays an important role in proving the lower semicontinuity of the cost functional. Hence, we shall only sketch them.

Fix $\alpha_n\in L^2(0,T)$ such that $\|\alpha_n\|_{2} \le C$ and $(y_n,\alpha_n)\in\Gamma^{\mu_n,{\rm opt}}[x_n]$.
Up to the extraction of a subsequence, we can assume that $\alpha_n$ converges weakly in $L^2(0,T)$ to some~$\alpha$ with $\|\alpha\|_{2} \le C$. Since $y_n$ converges uniformly to $y$, this implies that $y\in W^{1,2}(0,T)$ with $(y,\alpha)$ fulfilling~\eqref{grugen} and $y(\cdot)\in \Sigma$. We have proved that  $ y\in\Gamma_C[x]$.

Now we wish to prove that $y$ is optimal for $J^\mu$, namely
\begin{equation}\label{eq:pre34}
J^\mu(x;(y,\alpha))\leq J^\mu(x;(\hat y,\hat \alpha))\qquad \forall (\hat y,\hat \alpha) \in\Gamma^\mu[x]
\end{equation}
where $\Gamma^\mu$ is defined in~\eqref{eq:42}. We split our arguments according to which part of assumption~$(ii)$ is in force.

If assumption~$(ii)$-($b$) holds true, then $y(s)=x$ for $s\in [0,T]$ is the unique admissible trajectory starting from~$x$ with finite cost; this implies $\Gamma^{\mu,{\rm opt}}[x]=\{(x,0)\}$. Since $(y,\alpha)\in W^{1,2}(0,T)\times L^{2}(0,T)$, the cost $J^\mu(x;(y,\alpha))$ is finite. Then $y(s)=x$  in $[0,T]$. Hence, choosing $\alpha^*=0$, the trajectory~$(y,\alpha^*)\in\Gamma^{\mu,{\rm opt}}[x]$.

Now consider that assumption~$(ii)$-($a$) holds true.
Fix any $\hat y\in\Gamma^\mu[x]$. Then there exists a sequence $(\hat y_n,\hat \alpha_n)_{n\in\N}$ such that $(\hat y_n,\hat \alpha_n)\in W^{1,2}([0,T])\times L^2(0,T)$,   $
\hat y_n \to \hat y$ uniformly in $[0,T]$ as $n\to\infty$, and $\|\hat \alpha_n\|_{2}\leq \|\hat \alpha\|_{2} +o_n(1)$, where $\lim_n o_n(1)=0$.
Since $y_n\in \Gamma_C^{\mu_n,{\rm opt}}[x_n]$,
\begin{equation}\label{eq:34}
J^{\mu_n}(x_n;(y_n,\alpha_n))\leq J^{\mu_n}(x_n;(\hat y_n,\hat \alpha_n)).
\end{equation}
By the lower semicontinuity of $J^\mu$ with respect to $\alpha$ and the continuity of $L$ with respect to $\mu $, we have
\begin{equation*}
J^\mu(x;(y,\alpha))\leq \liminf_{n\to\infty}  J^{\mu_n}(x_n;(y_n,\alpha_n)).
\end{equation*}
Arguing as in~\cite[Proof of Lemma 3.4 (point $(b)$)]{CC}, we obtain
\begin{equation*}
\limsup_{n}J^{\mu_n}(x_n;(\hat y_n,\hat \alpha_n))\leq J^{\mu}(x;(\hat y,\hat \alpha)).
\end{equation*}
In conclusion, substituting these inequalities into~\eqref{eq:34} yields
$
J^\mu(x;(y,\alpha))\leq J^\mu(x;\break(\hat y,\hat\alpha))$.
By the arbitrariness of $(\hat y,\hat \alpha)\in\Gamma^\mu[x]$, we complete the proof of relation~\eqref{eq:pre34}.
\end{proof}
\begin{proof}[Proof of Theorem~\ref{sec:setting-notation-6}]
This proof follows the same arguments as in the proofs of~\cite[Theorem 3.5]{24SIAM} and of~\cite[Theorem 3.1]{CC}: We shall only sketch it.

Let us first recall that Remark~\ref{prp:ex_OT} ensures that, for every $x\in \rm{supp} (m_0)$,
$\GammamuoptC[x]\break\ne\emptyset$, where $\GammamuoptC[x]$ is defined as in \eqref{eq:42new}. Moreover, the set~$\Gamma_C$ is compact and, by Prokhorov' theorem \cite[Theorem 5.1.3]{AGS}, $\cP(\Gamma_C)$ is also compact.

We define the multivalued map
$E:\cP_{m_0}(\Gamma_C)\rightrightarrows \cP_{m_0}(\Gamma_C)$
as follows:
\begin{equation*}\label{eq:mappa}
E(\mu)=\{\hat \mu\in\cP_{m_0}(\Gamma_C)\;:\; \textrm{supp}\,\hat \mu_x\subset \GammamuoptC[x]\quad m_0-\textrm{a.e. } x\in\Sigma\},
\end{equation*}
where $\{\hat \mu_x\}_{x\in\Sigma}$ is the family of Borel probability measures on $\cP_{m_0}(\Gamma_C)$ obtained by the disintegration theorem \cite[Theorem 5.1.3]{AGS}.
Following the arguments of~\cite[Theorem 3.5]{24SIAM}, we obtain
that, for every $\mu\in\cP_{m_0}(\Gamma_C)$, the set $E(\mu)$ is not empty and convex
and that the map~$E$ has the closed graph property. In the proof of the latter property, we use the Kuratowski theorem and we replace  \cite[Proposition 3.10]{24SIAM} with our Lemma~\ref{prop2}.
In conclusion, the existence of a relaxed equilibrium is obtained by applying  Kakutani's theorem.
\end{proof}
\begin{remark}
By easy adaptations, our results hold also in the case $\nu=0$ (i.e., $\dot y=\alpha$, where dynamics are strongly controllable). Following   Remark~\ref{rmk:k=0}, the existence of a relaxed equilibrium is obtained under assumptions weaker than those of~\cite{CC} where $\Sigma$ is the closure of a bounded open set with $C^2$ boundary.
\end{remark}

\subsection{Mild solution}

Since Theorem~\ref{sec:setting-notation-6} ensures the existence of a relaxed equilibrium, we can now introduce the definition of a {\it mild} solution and establish its existence as an easy consequence of Theorem~\ref{sec:setting-notation-6}.
\begin{definition}
Let $\mu\in\cP_{m_0}(\Gamma_C)$ be a relaxed equilibrium.
The pair $(u,m)$ is the associated mild solution if $u$ is the value function  associated with $\mu$ by \begin{equation}\label{eq:4.1}
  u(x,t)= \inf_{ y\in \Gamma^\mu_t[x]}
  J^\mu_t(x; y ),
\end{equation}
where $J^\mu_t(x;y)=\int_t^T \left(L[m^\mu(\tau)] (y(\tau),\tau)+\frac{|\alpha(\tau)|^2}{2} \right)d\tau+ G[ m^\mu(T)](y(T))$,
$\Gamma^\mu_t[x]$ is the set~$\Gamma_t[x]$ with cost~$J^\mu_t$ and
$m\in C([0,T];\cP_1(\Sigma))$ is defined by
$m(t)=e_t\#\mu$.
\end{definition}
A straightforward consequence of Theorem~\ref{sec:setting-notation-6} and of Theorem~\ref{thm:u_cont2} is the following result.
\begin{corollary}
Under the assumptions of Theorem~\ref{sec:setting-notation-6} and of Theorem~\ref{thm:u_cont2}, there exists a mild solution~$(u,m)$ with $m\in C^{1/2}([0,T];\cP_1(\Sigma))$ and $u\in C(\Sigma\times [0,T])$.
\end{corollary}
\begin{remark}
We refer the reader to Section~\ref{sect:ex} for examples where the dynamics and the set~$\Sigma$ fulfill our hypotheses and we obtain the existence of a mild solution.
\end{remark}

\begin{remark}
As in classical cases, we obtain the uniqueness of the mild solution under some monotonicity property. We say that $F:\Sigma\times\cP(\Sigma)\to \R$ is monotone if, for any $m_1,m_2\in \cP(\Sigma)$, there holds $$\int_{\Sigma} (F(x,m_1)-F(x,m_2))(m_1-m_2)(dx)\ge 0.$$ 
Furthermore, we say that~$F$ is strictly monotone when  
$$\int_{\Sigma} (F(x,m_1)-F(x,m_2))(m_1-m_2)(dx)=0 {\text{ if and only if }} F(\cdot,m_1)\equiv F(\cdot,m_2).$$
 If for all $t$, $(x,m)\mapsto L[m](x,t)$ and $(x,m)\mapsto G[m](x)$ are strictly monotone, then it can be proved, with the same arguments as in \cite[Theorem 4.1 and Remark 4.1]{CC}, that if $(u_1,m_1)$ and $(u_2,m_2)$ are mild solutions respectively associated with two relaxed equilibria $\mu_1$ and $\mu_2$,  then $u_1=u_2$. Under a more restrictive monotonicity assumption on $L$, it can also be proved that $m_1=m_2$.

It is worth noticing that the uniqueness of the mild solution does not imply the uniqueness of the relaxed MFG equilibrium as shown in \cite[Example 3.3]{24SIAM}.
\end{remark}

\section*{Acknowledgments}
The authors warmly thank the referees for several suggestions and fruitful comments. A. Cutr\`i thanks l'Institut de Recherche Math\'ematique de Rennes (IRMAR) of the University of Rennes for the invitation and the kind hospitality in the period March-May 2026.\\
The first author is member of the GNAMPA-INDAM and she is  partially supported by the MUR Excellence Department Project 2023-2027 MatMod@Tov (CUP E83C23000330006) awarded to the Department of Mathematics, University of Rome Tor Vergata and by the Project ``Stability in Analysis and Dynamics" of the University of Rome Tor Vergata CUP E83C25000590005.
The second and the third author are members of the GNAMPA-INDAM and are partially supported by the INDAM-GNAMPA project CUP E53C25002010001 and also by the PRIN 2022 project ``PDEs and optimal control methods in mean field games, population dynamics and multi-agent models'' and by the KAUST CRG Project ``Free Boundary Problems, Mean Field Games, Crowd Motion and Lipschitz Learning: The Infinity-Laplacian in Action''.
The fourth author is partially supported by ANR (Agence Nationale de la Recherche) through project COSS, ANR-22-CE40-0010-01.

\end{document}